\documentclass[
a4paper,
12pt,
]{article}

\usepackage{amsmath,amssymb,amsthm}
\usepackage{paralist,enumitem,bm,placeins,url}
\usepackage{color,graphicx,subfig}
\usepackage[margin=30mm,top=40mm]{geometry}
\usepackage[font=small]{caption}
\usepackage[numbers,sort&compress]{natbib}
\usepackage{diagbox}

\usepackage{algorithm}
\usepackage{algpseudocode}

\usepackage{array}
\newcolumntype{C}[1]{>{\centering\arraybackslash}p{#1}}
\usepackage{tocbasic}
\DeclareTOCStyleEntry[
beforeskip=.2em plus 1pt,
]{tocline}{section}
\usepackage{tikz}

\usepackage{hyperref}
\makeatletter
\hypersetup{%
colorlinks=true,
linkcolor=blue,%
citecolor=red,%
urlcolor=blue
}
\makeatother

\usepackage[normalem]{ulem}
\usepackage{mathrsfs}
\usepackage[title]{appendix}

\newtheorem{thm}{Theorem}[section]
\newtheorem{lem}[thm]{Lemma}
\newtheorem{rem}[thm]{Remark}

\newcommand*{\vv}[1]{\vec{\mkern0mu#1}}

\newcommand{\norm}[1]{\Vert#1\Vert}

\newcommand{\bR}{{\mathbb R}}
\newcommand{\bN}{{\mathbb N}}

\newcommand{\bU}{\mathbb{U}}
\newcommand{\bP}{\mathbb{P}}
\newcommand{\tD}{\mathbb{D}}
\newcommand{\mX}{\mathscr{X}}

\newcommand{\bkap}{{\bar{\varkappa}}}
\newcommand{\Gt}{{\Gamma(t)}}
\newcommand{\Gm}{{\Gamma^m}}
\newcommand{\Dtw}{\partial_t^{\vec w}}
\newcommand{\Dtv}{\partial_t^{\vec v}}
\newcommand{\vol}{\operatorname{vol}}
\def\widebar{\overline}

\newcommand{\dL}{{\rm d}\mathscr{L}}
\newcommand{\dH}{{\rm d}\mathscr{H}}
\newcommand{\rd}{\;{\rm d}}
\newcommand{\id}{{\rm id}}
\newcommand{\dd}[1]{\frac{\rm d}{{\rm d}#1}}
\newcommand{\ddt}{\dd{t}}

\newcommand{\nn}{\nonumber}
\newcommand{\ttau}{\Delta t}
\newcommand{\ipd}[1]{\big\langle#1\big\rangle}
\newcommand{\ipD}[1]{\left(#1\right)}
\newcommand{\nabs}{\nabla_{\!\!s}}

\newcommand{\Nbulk}{\mat{\vec{N}}_{\Gamma,\Omega}}

\newcommand{\mat}[1]{\uuline{#1}\rule{0pt}{0pt}}
\newcommand{\Id}{{I\!d}}

\numberwithin{equation}{section}

{{\upshape\bfseries AMS subject classifications. }\ignorespaces}{}

\begin{document}

\renewcommand{\thefootnote}{\fnsymbol{footnote}}

\title{A unified energy-stable finite element approximation for evolving fluidic biomembranes}

\author{Harald Garcke\thanks{Fakult{\"a}t f{\"u}r Mathematik, Universit{\"a}t Regensburg,
93040 Regensburg, Germany ({\color{blue}harald.garcke@ur.de})}
\and Robert N\"urnberg\thanks{Dipartimento di Matematica, Universit\`a di Trento,
38123 Trento, Italy ({\color{blue}robert.nurnberg@unitn.it})}
\and Quan Zhao\thanks{School of Mathematical Sciences, University of Science and Technology of China, 230026 Hefei, Anhui, China ({\color{blue}quanzhao@ustc.edu.cn})}
}

\date{}
\maketitle

\begin{abstract}
\noindent We present a unified finite element method for the dynamics of fluidic biomembranes. The model is governed by the Navier--Stokes equations in the bulk coupled to the surface Navier--Stokes equations on the evolving biomembrane surface, with bending forces arising from the Willmore energy. By allowing the bulk mesh velocity to be independent of the fluid velocity and permitting a free tangential surface velocity, we are able to derive a unified weak formulation of the coupled bulk-surface Navier--Stokes system. To address the bending force, we consider an evolution equation for the curvature and propose a surface arbitrary Lagrangian--Eulerian (ALE) weak formulation. Discretization with either fitted or unfitted finite elements leads to
well-posed fully discrete linear schemes that are unconditionally energy 
stable. We present a variety of numerical examples to demonstrate the favourable properties of the proposed methods.

\end{abstract}



\setlength\parskip{1ex}
\renewcommand{\thefootnote}{\arabic{footnote}}

\setcounter{equation}{0}

\setlength\parindent{24pt}

\section{Introduction}\label{sec:intro}

Fluidic biomembranes play an essential role in many biological processes, including vesicle transport, cell motility, and membrane-mediated interactions between proteins. These membranes are typically composed of lipid bilayers that exhibit fluid-like behaviour along the membrane surface while interacting dynamically with the surrounding viscous fluid. At the continuum level, the membrane can therefore be modelled as an incompressible viscous surface endowed with bending elastic energies \cite{Seifert97}. The bending elasticity in its simplest version is commonly described by the Willmore energy \cite{Willmore93},
\begin{equation}
E_b = \frac{1}{2}\int_{\Gamma}\varkappa^2\dH^{d-1},
\end{equation}
where $\Gamma$ is a hypersurface in $\bR^d$, $d\in\{2,3\}$, $\varkappa$ denotes its mean curvature, and $\dH^{d-1}$ represents integration with respect to the $(d-1)$-dimensional Hausdorff measure in $\bR^d$. This gives rise to nonlinear geometric forces acting on the membrane (see \eqref{eq:fGamma}, below). As a consequence, the resulting model leads to a coupled bulk-surface system consisting of the incompressible Navier--Stokes equations in the bulk and the surface incompressible Navier--Stokes equations on the evolving membrane, together with highly nonlinear fourth-order geometric forces \cite{Arroyo09relaxation, BGN16sfluidic}. The presence of these high-order geometric forces, together with the strong coupling between membrane motion and bulk fluid dynamics, poses significant challenges for the numerical approximation of evolving fluidic biomembranes.

Since the geometric forces are derived from the Willmore energy, their discretization is closely related to the numerical approximation of the Willmore flow and its volume- and area-preserving variant, the Helfrich flow. These include parametric finite element methods (FEM) \cite{Rusu05, Dziuk08, BGN08willmore, Bonito10parametric, pwfade, KovacsLL21, Hu22evolving, Kemmochi25structure, RumpfSS25preprint, GNZ25willmore, GNZ26willmore}, phase-field approximations \cite{Du04phase, FrankenRW13}, the thresholding method \cite{Esedoglu08threshold}, and the level set approach \cite{Droske04level}. For the numerical approximation of the dynamics of lipid membranes and vesicles in fluidic environments, there exists a large body of work (see \cite{Arroyo09relaxation, Franke11numerical, Bonito11dynamics, Salac12reynolds, Kruger13crossover, Laadhari14computing, shi2014three, Hu14immersed, Aland14diffuse, Rodrigues15semi, BGN16sfluidic, BGN17fluidic, Krause23jcp, Garcke26PFEM} and the references therein). Despite these contributions, to the best of our knowledge, numerical analysis for fluidic biomembrane models remains very limited, particularly with respect to stability estimates, due to the complexity of the model. We also refer the reader to \cite{BGN16sfluidic, BGN17fluidic} for an unfitted FEM for a dynamic model of fluidic biomembranes, where a stability estimate for the semi-discretization is available. These stable approximations of the geometric forces at the semi-discrete level build on the pioneering work of Dziuk for the Willmore flow \cite{Dziuk08}. In this work, we focus on the sharp-interface model of fluidic biomembrane dynamics that was considered in \cite{BGN16sfluidic} and propose a finite element approximation in a unified framework. In particular, we close an existing gap in the literature by introducing, to the best of our knowledge, the first fully discrete finite element scheme that preserves the intrinsic energy structure of the system and admits a discrete energy stability estimate. 

\textit{The first motivation} for the present work stems from our recent numerical study of the Willmore flow in \cite{GNZ25willmore, GNZ26willmore}, where a fully discrete energy-stable approximation was presented. The introduced scheme splits the normal and tangential components of the velocity by approximating the gradient flow structure of Willmore flow through an evolution equation for the mean curvature, in which a convective term involving the tangential velocity arises naturally (see \eqref{eq:tvark}, below). The mean curvature evolution equation can be treated in a novel way using a surface-based arbitrary Lagrangian--Eulerian (ALE) technique, which gives a fully discrete, linear and unconditionally energy-stable approximation that allows for an independent choice of the tangential velocity. This is analogous to the classical ALE approach for moving bulk domains. For example, we recall that for two-phase flow the present authors have recently been able to devise an energy-stable approximation that is independent of the bulk mesh velocity, see \cite{GNZ23, GNZ23asy}.

\textit{The second motivation} is that the surface ALE technique in \cite{GNZ26willmore} can also be employed to handle the surface Navier--Stokes equations on the evolving membrane, leading to a fully discrete FEM approximation with a stability estimate for the kinetic energy of the surface fluid. This provides a natural generalization of the ideas developed for Willmore-type flows in \cite{GNZ26willmore} to the case of surface Navier--Stokes equations on evolving manifolds. 

\textit{The third motivation} concerns the treatment of the coupled bulk--surface Navier--Stokes system, where we propose a unified weak formulation that allows the bulk mesh velocity to be completely independent of the fluid velocity. This accommodates the unfitted approach by selecting a zero bulk mesh velocity \cite{BGN15stable} or alternatively the fitted approach, which allows for some flexibility in choosing the bulk mesh velocity within the interior domain \cite{Duan22energy, GNZ23, Hu26ALE}.

In this setting, we unify the treatment of the geometric forces and the coupled bulk-surface Navier--Stokes equations, which enables novel fitted and unfitted FEM approximations that respect the intrinsic energy structure of fluidic biomembrane dynamics.

The remainder of this paper is organized as follows. In Section~\ref{sec:model}, we introduce the sharp-interface model for the dynamics of fluidic biomembranes, formulated as a coupled bulk-surface system. Section~\ref{sec:weakform} presents a unified weak formulation of the model, including the bulk-surface Navier--Stokes equations, a parametric representation of the geometric forces, and the evolution of the interface with a tangential velocity. In Section~\ref{sec:FEM}, we develop an ALE fitted mesh approach and present a linear finite element approximation of the weak formulation, together with proofs of solvability and an energy stability estimate. Numerical results based on the ALE approach are reported in Section~\ref{sec:num}. We then turn to an unfitted finite element approximation in Section~\ref{sec:unf}, where we present the fully discrete scheme, state well-posedness and unconditional stability and present some numerical results. Finally, conclusions are drawn in Section~\ref{sec:con}.

\section{The sharp-interface model}
\label{sec:model}

We follow the notations in \cite{BGN16sfluidic} and consider the dynamics of two fluids in the
domain $\Omega\subset\bR^d$ with
$\Omega = \Omega_+(t)\cup\Omega_-(t)\cup\Gamma(t)$, where  $d\in\{2,3\}$,
see Fig.~\ref{fig:model}.
Here $\Omega_\pm(t)$ are the regions occupied by the two fluids with densities $\rho_\pm\geq 0$ and viscosities $\mu_\pm>0$, respectively. The two fluids are separated by the interface
$\Gamma(t)$. Let $\vec u:\Omega\times[0,T]\to\bR^d$
be the fluid velocity and $p:\Omega\times[0, T]\to\bR$ be the pressure. The dynamic system is then governed by the standard incompressible Navier--Stokes equations %
\begin{subequations}\label{eqn:nv}
\begin{alignat}{2}
\label{eq:nv}
\rho_\pm\partial_t^\bullet\vec u - \nabla\cdot\mat{\sigma} &= 0\qquad &&\mbox{in}\quad\Omega_\pm(t),\\
\label{eq:div}
\nabla\cdot\vec u &= 0\qquad &&\mbox{in}\quad\Omega_\pm(t),\\
\label{eq:BD1}
\vec u &= \vec g\quad &&\mbox{on}\quad\partial_{_1}\Omega,\\
\mat{\sigma}\,\vec n_{_{\partial\Omega}} &=\vec 0\quad &&\mbox{on}\quad\partial_{_2}\Omega,
\label{eq:BD2}
\end{alignat}
\end{subequations}
where $\partial_t^\bullet$ is the material time derivative such that
\begin{equation}\label{eq:materT}
\partial_t^\bullet\vec\varphi  = \partial_t\vec\varphi + (\vec u\cdot\nabla)\vec\varphi,\qquad\mbox{for any}\quad \vec\varphi: \Omega\times[0,T]\to\bR^d.
\end{equation}
In addition, we have the stress tensor
\[\mat{\sigma}=2\mu_\pm\mat{D}(\vec u)-p\,\mat{\Id}=\mu_\pm[\nabla\vec u + (\nabla\vec u)^T]-p\mat{\Id}\quad\mbox{in}\quad\Omega_\pm(t),\]
where $\mat{D}(\vec u)=\frac{1}{2}\left(\nabla\vec u + (\nabla\vec u)^T\right)$ is the rate of deformation tensor, with $\nabla\vec u=(\partial_{x_j}u_i)_{i,j=1}^d$ and $\vec u = (u_1,\ldots, u_d)^T$, and $\mat{\Id}\in\bR^{d\times d}$ denotes the identity matrix. Furthermore, we denote by $\partial\Omega=\partial_1\Omega\cup\partial_2\Omega$ a disjoint partitioning of the boundary of $\Omega$ and assume that $\mathscr{H}^{d-1}(\partial_1\Omega)>0$. We also assume that $\partial\Omega\cap\partial\Omega_-(t)=\emptyset$ so that $\Gamma(t)$ is a closed hypersurface and does not intersect $\partial\Omega$. Here, we prescribe an inhomogeneous boundary condition on $\partial_1\Omega$ and a stress-free condition on $\partial_2\Omega$ with $\vec n_{_{\partial\Omega}}$ standing for the outwards unit normal to $\partial\Omega$. In the case where $\partial_2\Omega=\emptyset$, i.e., $\partial_1\Omega=\partial\Omega$, we require the compatibility condition
\begin{equation}
\int_{\partial\Omega}\vec g\cdot\vec n_{\partial\Omega}\dH^{d-1}=0,
\end{equation}
which is necessary for the solvability of the divergence equation.

\begin{figure}
\centering
\begin{tikzpicture}[scale=1.2]
\definecolor{innerblue}{RGB}{180,210,235}
\definecolor{outerorange}{RGB}{245,210,170}

\fill[outerorange] (-3,-2) rectangle (3,2);

\fill[innerblue] (0,0) circle (1);
\draw[thick] (0,0) circle (1);
\draw[thick] (-3,-2) rectangle (3,2);
\node at (0,0) {$\Omega_{-}(t)$};
\node at (2,-1.3) {$\Omega_{+}(t)$};
\node at (1.4,0.3) {$\Gamma(t)$};
\draw[->,thick,blue] (0.7,0.7) -- (1.2,1.2);
\node at (1.00,1.3) {$\vec{\nu}$};
\end{tikzpicture}
\caption{A schematic illustration of the two-phase flow in $\Omega_\pm(t)$ with a biomembrane $\Gamma(t)$. }
\label{fig:model}
\end{figure}
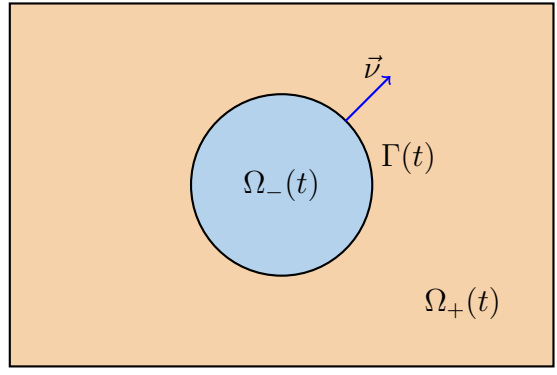

We assume that a parameterization of the evolving interface $\Gamma(t)$ is given by
\begin{equation}
\vec x(\cdot, t): \Upsilon\times[0,T]\to\bR^d,
\end{equation}
where $\Upsilon$ is a suitable reference surface without boundary in $\bR^d$. The surface velocity induced by this parameterization is thus defined as
\begin{equation}
\label{eq:para}
\vec v(\vec x(\vec q,t), t) = \partial_t\vec x(\vec q,t)\qquad\forall\vec q\in\Upsilon. 
\end{equation}
The mean curvature of the surface $\Gamma(t)$, denoted by $\varkappa$, can be defined via the following identity:
\begin{equation}
\varkappa\,\vec\nu  = \Delta_s\vec\id\qquad\mbox{on}\quad\Gamma(t),
\end{equation}
where $\vec\id$ is the identity function, $\vec\nu$ is the unit normal to $\Gamma(t)$, and $\Delta_s=\nabs\cdot\nabs$ is the Laplace--Beltrami operator on the surface with $\nabs$ being the surface gradient.

On the free surface $\Gamma(t)$, the following conditions need to hold
\begin{subequations}\label{eqn:ifc}
\begin{alignat}{2}
\label{eq:if1}
[\vec u]^+_-&=\vec 0\qquad &&\mbox{on}\quad\Gamma(t), \\
\label{eq:if2}
\rho_\Gamma\partial_t^\bullet\vec u- \nabs\cdot\mat{\sigma_\Gamma} &=[\mat{\sigma}\,\vec\nu]_-^+  +\alpha f_\Gamma\,\vec\nu\quad &&\mbox{on}\quad\Gamma(t), \\
\label{eq:if3}
\nabs\cdot\vec u &= 0\qquad &&\mbox{on}\quad\Gamma(t), \\
\label{eq:if4}
\vec v\cdot\vec\nu &= \vec u\cdot\vec\nu\quad &&\mbox{on}\quad\Gamma(t),
\end{alignat}
\end{subequations}
where $[\cdot]_-^+$ denotes the jump value across $\Gamma(t)$ from $\Omega_-(t)$ to $\Omega_+(t)$, $\rho_\Gamma\in\bR_{\geq 0}$ denotes the surface material density and $\partial_t^\bullet$ again denotes the material time derivative
\begin{equation}\label{eq:fmatT}
\partial_t^\bullet\vec\zeta = \partial_t\vec\zeta + (\vec u\cdot\nabla)\vec\zeta,\quad\mbox{for all}\quad\vec\zeta\in [H^1(\mathcal{G}_T)]^d,
\end{equation}
where we introduce the space-time surface
\[\mathcal{G}_T:=\cup_{t\in[0,T]}\Gamma(t)\times\{t\}.\]
The material time derivative in \eqref{eq:fmatT} is well defined and in fact depends only on the values of $\vec\zeta$ on $\mathcal{G}_T$.
The surface stress tensor $\mat{\sigma_\Gamma}$ is given by
\begin{equation}
\mat{\sigma_\Gamma} = 2\mu_\Gamma \mat{D}_s(\vec u) - p_{_\Gamma}\mat{P}_\Gamma\qquad\mbox{on}\quad\Gamma(t),
\end{equation}
where $\mat{P}_\Gamma$ is the projection operator onto the tangential space of $\Gamma(t)$,
\begin{equation}
\mat{P}_\Gamma = \mat{\Id} - \vec\nu\otimes\vec\nu \qquad\mbox{on}\quad\Gamma(t),
\end{equation}
and the surface deformation tensor is
\begin{equation}
 \mat{D}_s(\vec u) = \frac{1}{2}\mat{P}_\Gamma(\nabs\vec u + (\nabs\vec u)^T)\mat{P}_\Gamma\qquad\mbox{on}\qquad\Gamma(t).
 \end{equation}

We call \eqref{eq:if2} and \eqref{eq:if3} the surface incompressible Navier--Stokes equations.
On the right-hand side of \eqref{eq:if2}, $\alpha\in\bR_{>0}$ is the bending rigidity, and $f_\Gamma$ is the curvature force given by
\begin{equation}
f_\Gamma = -\Delta_s\varkappa -(\varkappa-\bkap) \,|\nabs\vec\nu|^2+\frac{1}{2}(\varkappa-\bkap)^2\varkappa,\qquad\mbox{on}\quad\Gt,\label{eq:fGamma}
\end{equation}
where $\nabs\vec\nu$ is the Weingarten map and $|\mat{A}|^2={\rm tr}(\mat{A}\mat{A}^T)$ is the Frobenius norm for any matrix $\mat{A}\in\bR^{d\times d}$. Note that in the case of $d=2$, $|\nabs\vec\nu|^2 = \varkappa^2$. Here, $f_\Gamma$ arises from the well-known Willmore energy
\begin{equation}\label{eq:Willmore}
E_{\rm \bkap}(\vec x(t), \varkappa(t)) = \frac{1}{2}\int_{\Gamma(t)}(\varkappa-\bkap)^2\,\rd\mathscr{H}^{d-1},
\end{equation}
where $\bkap\in\bR$ is the spontaneous curvature.  The value of $\bkap$ takes into account that the surface may have a preferred mean curvature, which plays a crucial role in biomembranes. In particular, the first variation of \eqref{eq:Willmore} for a closed hypersurface $\Gamma(t)$ leads to (see \cite[(222)]{Barrett20})
\begin{equation}
\ddt E_{\bkap}(\vec x(t), \varkappa(t)) = -\int_{\Gamma(t)}f_\Gamma\,\vec\nu\cdot\vec v\dH^{d-1}.
\end{equation}

We further introduce the density and viscosity functions as
\begin{equation}\label{eq:denvisf}
\rho(\cdot, t) = \rho_-\mX_{\Omega_-(t)} + \rho_+\mX_{\Omega_+(t)},\qquad \mu(\cdot, t) = \mu_-\mX_{\Omega_-(t)} + \mu_+\mX_{\Omega_+(t)},
\end{equation}
where $\mX_E$ is the characteristic function of a set $E$. To form a complete system, we also include the initial conditions
\begin{equation}
\Gamma(0)=\Gamma_0,\qquad \rho\,\vec u(\cdot,0) = \rho\,\vec u_0\quad\mbox{in}\quad\Omega,\qquad \rho_\Gamma\,\vec u(\cdot,0) = \rho_\Gamma\,\vec u_0\quad\mbox{on}\quad\Gamma_0.
\end{equation}
The complete sharp-interface model for the evolving fluidic biomembrane consists of the bulk Navier--Stokes equations with boundary conditions in \eqref{eqn:nv}, alongside the surface Navier--Stokes equations and interface conditions in \eqref{eqn:ifc}. The total free energy of the system is the sum of the bulk and surface fluid kinetic energies as well as the bending energy
\begin{equation} \label{eq:Et}
E(t)= \frac{1}{2}\int_{\Omega}\rho\,|\vec u|^2\dL^d + \frac{\rho_\Gamma}{2}\int_{\Gamma(t)}|\vec u|^2\dH^{d-1} + \frac{\alpha}{2}\int_{\Gamma(t)}\,(\varkappa-\bkap)^2\dH^{d-1}. 
\end{equation}
We have the following lemma for the strong solution of the system. 

\begin{lem} \label{lem:dtE}
\begin{align}
&\ddt E(t) +2\int_{\Omega}\mu |\mat{D}(\vec u)|^2\dL^d + 2\int_{\Gamma(t)}\mu_\Gamma|\mat{D}_s(\vec u)|^2\dH^{d-1}\nn\\
&\qquad\quad +\frac{\rho_+}{2}\int_{\partial\Omega}(\vec u\cdot\vec n)|\vec u|^2\dH^{d-1}=  \int_{\partial_1\Omega}(\mat{\sigma}\,\vec n)\cdot\vec g\,\dH^{d-1}.
\end{align}
Furthermore, the system satisfies the volume preservation and surface area preservation
\begin{subequations}\label{eqn:strongcon}
\begin{align}
&\ddt\vol(\Gt)=\int_{\Gt}\vec v\cdot\vec\nu\dH^{d-1}=0,\\
&\ddt|\Gt|=\int_{\Gt}\nabs\cdot\vec u\,\dH^{d-1}=0.
\end{align}
\end{subequations}
\end{lem}

\begin{proof}
It follows from the Reynolds transport theorem that
\begin{subequations}
\begin{align}\label{eqn:transp}
\dfrac12\ddt \int_{\Omega_+(t)}\rho_+|\vec u|^2\,\dL^d= \int_{\Omega_+(t)}\rho_+\partial_t\vec u\cdot\vec u\,\dL^d- \frac{\rho_+}{2}\int_{\Gamma(t)}|\vec u|^2(\vec u\cdot\vec\nu)\dH^{d-1},\\
\dfrac12\ddt \int_{\Omega_-(t)}\rho_-|\vec u|^2\,\dL^d= \int_{\Omega_-(t)}\rho_-\partial_t\vec u\cdot\vec u\,\dL^d + \frac{\rho_-}{2}\int_{\Gamma(t)}|\vec u|^2(\vec u\cdot\vec\nu)\dH^{d-1}
\end{align}
\end{subequations}
Multiplying \eqref{eqn:nv} by $\vec u$, integrating over $\Omega_-(t)$ and
$\Omega_+(t)$, and using \eqref{eqn:transp},  \eqref{eq:div},
\eqref{eq:if1} and \eqref{eq:if4}, yields
\begin{align}
&\frac{1}{2}\ddt\int_\Omega \rho\,|\vec u|^2\dL^d
 +2\int_\Omega\mu|\mat{D}(\vec u)|^2\dL^d+\frac{\rho_+}{2}\int_{\partial\Omega}
 (\vec u\cdot\vec n_{_{\partial\Omega}})|\vec u|^2\dH^{d-1}\nn\\
& =
-\int_{\Gamma(t)}[\mat{\sigma}\,\vec\nu]_-^+\cdot\vec u\,\dH^{d-1}
 +\int_{\partial_1\Omega}(\mat{\sigma}\,\vec n_{_{\partial\Omega}})\cdot\vec g
 \,\dH^{d-1}.
\label{eq:bulkenergystrong}
\end{align}
where we used
\begin{subequations}\label{eqn:convecint}
\begin{align}
&\int_{\Omega_+(t)}\rho_+(\vec u\cdot\nabla)\vec u\cdot\vec u\dL^d = \int_{\Omega_+(t)}\rho_+\nabla\cdot\left(\dfrac12|\vec u|^2\vec u\right)\dL^d\nn\\
 &\qquad\qquad = \frac{\rho_+}{2}\int_{\partial\Omega}|\vec u|^2(\vec u\cdot\vec n_{\partial\Omega})\dH^{d-1} - \frac{\rho_+}{2}\int_{\Gt}|\vec u|^2(\vec u\cdot\vec\nu)\dH^{d-1}\\
&\int_{\Omega_-(t)}\rho_-(\vec u\cdot\nabla)\vec u\cdot\vec u\dL^d  = \frac{\rho_-}{2}\int_{\Gt}|\vec u|^2(\vec u\cdot\vec\nu)\dH^{d-1}.
\end{align}
\end{subequations}

Next, testing the surface momentum balance \eqref{eq:if2} with $\vec u$ and
integrating by parts on the closed surface $\Gamma(t)$ gives
\begin{align}
&\frac{\rho_\Gamma}{2}\ddt\int_{\Gamma(t)}|\vec u|^2\dH^{d-1}
 +2\int_{\Gamma(t)}\mu_\Gamma|\mat{D}_s(\vec u)|^2\dH^{d-1}\nn\\
&\qquad =
\int_{\Gamma(t)}[\mat{\sigma}\,\vec\nu]_-^+\cdot\vec u\,\dH^{d-1}
 +\alpha\int_{\Gamma(t)}f_\Gamma\,\vec\nu\cdot\vec u\,\dH^{d-1}.
\label{eq:surfaceenergystrong}
\end{align}
In this step, $\nabs\cdot\vec u=0$ eliminates the surface pressure
contribution, and the transport theorem on $\Gt$ together with \eqref{eq:if3} and
\eqref{eq:if4} gives the time derivative of the surface kinetic energy.
Moreover, by \eqref{eq:if4} and the first variation of the bending energy,
\begin{equation}\label{eq:bendingenergystrong}
\alpha\int_{\Gamma(t)}f_\Gamma\,\vec\nu\cdot\vec u\,\dH^{d-1}
=\alpha\int_{\Gamma(t)}f_\Gamma\,\vec\nu\cdot\vec v\,\dH^{d-1}
=-\alpha\,\ddt E_{\bkap}(\vec x(t),\varkappa(t)).
\end{equation}
Adding \eqref{eq:bulkenergystrong}, \eqref{eq:surfaceenergystrong} and \eqref{eq:bendingenergystrong} then gives the energy dissipation identity.

For the conservation laws \eqref{eqn:strongcon}, we refer the reader to \cite{BGN16sfluidic, BGN17fluidic} and also to the proof in Theorem~\ref{thm:weak} for the weak setting.
 
\end{proof}

\section{Unified weak formulation}\label{sec:weakform}

In this section, we propose a novel weak formulation for the whole system by decomposing it into three subsystems. The first subsystem corresponds to the coupled bulk-surface Navier--Stokes equations. The second subsystem addresses the curvature force $f_\Gamma$ defined in \eqref{eq:fGamma}, while the third subsystem is for the evolution of the interface with tangential velocity.

We will derive the weak formulations of the three subsystems in a more general framework by considering the case of a moving coordinate system. Let
\begin{equation}
\vec\Phi[t]: \Omega\to\Omega,\qquad \vec y\mapsto \vec\Phi[t](\vec y)= \vec x(\vec y,t)\quad\mbox{for all}\quad \vec y\in\Omega,\quad t\in[0,T],\label{eq:Phi}
\end{equation}
be an ALE map of the domain, where $\vec\Phi[t]\in [W^{1,\infty}(\Omega)]^d$ is invertible with $\vec\Phi[t]^{-1}\in [W^{1,\infty}(\Omega)]^d$. This induces the domain mesh velocity
\begin{equation*}
\vec w(\vec x, t):=\frac{\partial\vec x(\vec y,t)}{\partial t}\left|\right._{\vec y = \vec\Phi[t]^{-1}(\vec x)}\quad\mbox{for all}\quad \vec y\in\Omega,\quad t\in[0,T].
\end{equation*}
For any vector field $\vec\varphi: \Omega\times[0,T]\to\bR^d$, we also define the time derivative with respect to the ALE moving frame
\begin{equation}\label{eq:ALET}
\Dtw\vec\varphi = \partial_t\vec\varphi + (\vec w\cdot\nabla)\vec\varphi.
\end{equation}
By comparison with \eqref{eq:materT}, we obtain that
\begin{equation}
\partial_t^\bullet\vec\varphi = \Dtw\vec \varphi + ([\vec u-\vec w]\cdot\nabla)\vec\varphi.\label{eq:dTd}
\end{equation}
In the case of $\vec u=\vec w$, the operator $\Dtw$ coincides with the material time derivative $\partial_t^\bullet$.

Similarly, given the interface velocity $\vec v$, satisfying the constraint \eqref{eq:if4} for geometric compatibility, we introduce the interface operator
\begin{equation}\label{eq:fALET}
\Dtv\zeta = \partial_t\zeta + (\vec v\cdot\nabla)\zeta\qquad\forall\zeta\in H^1(\mathcal{G}_T).
\end{equation}
Now, comparing \eqref{eq:fALET} with \eqref{eq:fmatT}, and recalling \eqref{eq:if4}, we obtain
\begin{equation}\label{eq:fdTd}
\partial_t^\bullet\zeta = \Dtv\zeta + ([\vec u- \vec v]\cdot\nabs)\zeta\qquad\forall\zeta\in H^1(\mathcal{G}_T).
\end{equation}

We note that $\vec w$ and $\vec v$ represent the velocities of the moving bulk domain and the moving interface, respectively. Our unified framework makes it possible to consider two common choices of the bulk domain velocity $\vec w$, each leading to very different finite element discretization methods:
\begin{itemize}
\item [(i)] leading to a fitted mesh approach:
\begin{subequations}
\begin{alignat}{2}
&\vec w(\vec x, t) = \vec v(\vec x, t)\qquad &&\mbox{on}\quad\Gamma(t),\\
&\vec w(\vec x, t)\cdot\vec n_{_{\partial\Omega}}(\vec x) = 0\qquad &&\mbox{on}\quad\partial\Omega,\label{eq:wbd}
\end{alignat}
\end{subequations}
with arbitrary choices of the mesh velocity $\vec w(\vec x, t)$ for $\vec x\in\Omega_\pm(t)$.
\item [(ii)] leading to an unfitted mesh approach:
\[\vec w(\vec x, t) = \vec 0\quad \mbox{for}\quad\vec x\in\widebar{\Omega},\]
which implies that the map defined in \eqref{eq:Phi} is the identity map in $\widebar{\Omega}$, i.e., $\vec\Phi[t]=\vec\id|_{\widebar{\Omega}}$.
\end{itemize}

\subsection{Weak formulation for bulk-surface NS equations}\label{sec:wfnv}

To introduce the weak formulation, we define the following (affine) function spaces:
\begin{subequations}\label{eqn:UPspaces}
\begin{align}
\mathbb{U}(\vec b)&:=\bigl\{\vec\varphi\in [H^1(\Omega)]^d:\;\vec\varphi = \vec b\;\;\mbox{on}\;\;\partial_1\Omega\bigr\},\\
\mathbb{V}(\vec b)&:=H^1(0,T;[L^2(\Omega)]^d)\cap L^2(0,T;\mathbb{U}(\vec b)),\\
\mathbb{V}_\Gamma(\vec b) &:= \bigl\{\vec\varphi\in\mathbb{V}(\vec b):\; \vec\varphi|_{\mathcal{G}_T}\in [H^1(\mathcal{G}_T)]^d\bigr\},
\end{align}
\end{subequations}
where we denote by $(\cdot,\cdot)$ the $L^2$--inner product over $\Omega$.
Moreover, we define
\begin{equation}
\mathbb{P}:= \left\{\begin{array}{ll}
L^2(\Omega)\qquad &\mbox{if}\quad \mathscr{H}^{d-1}(\partial_2\Omega)>0,\\
\bigl\{\eta\in L^2(\Omega):\;(\eta,~1) = 0\bigr\}\quad &\mbox{otherwise}.
\end{array}\right.
\end{equation}

We first consider the material time derivative of the bulk fluid velocity in \eqref{eq:nv}. To this end, we multiply it with a test function $\vec\chi\in\mathbb{V}(\vec 0)$ and integrate over $\Omega$. We let $\bigl(\cdot,\cdot)_{\Omega_\pm}$, $\ipd{\cdot,\cdot}_{\Gt}$ and $\ipd{\cdot,\cdot}_{\partial_2\Omega}$ be the $L^2$--inner product over $\Omega_\pm(t)$, $\Gt$ and $\partial_2\Omega$, respectively. Then, on recalling \eqref{eq:dTd}, we have that
\begin{align}
\bigl(\rho\,\partial_t^\bullet\vec u,~\vec\chi\bigr) &= \bigl(\rho\,\Dtw\vec u,~\vec\chi\bigr)+\bigl(\rho\,([\vec u-\vec w]\cdot\nabla)\vec u,~\vec\chi\bigr)\nn\\
&=\frac{1}{2}\bigl(\rho\,\Dtw\vec u,~\vec\chi\bigr) + \frac{1}{2}\bigl(\rho\,\Dtw\vec u,~\vec\chi\bigr)+\bigl(\rho\,([\vec u-\vec w]\cdot\nabla)\vec u,~\vec\chi\bigr)\nn\\
&=\frac{1}{2}\left[\ddt\bigl(\rho\,\vec u,~\vec\chi\bigr) + \bigl(\rho\,\Dtw\vec u,~\vec\chi\bigr) - \bigl(\rho\,\vec u,~\Dtw\vec\chi\bigr) +\rho_+\ipd{\vec u\cdot\vec n_{_{\partial\Omega}},~\vec u\cdot\vec\chi}_{\partial_2\Omega}\right]\nn\\
&\qquad + \mathscr{A}(\rho\,[\vec u-\vec w];\vec u,\vec\chi),
\label{eq:ale2}
\end{align}
where the third equality is due to \eqref{eq:ale1} in the appendix, and we also introduced the skew-symmetric term
\begin{equation*}
\mathscr{A}(\vec\zeta;\vec u,\vec\chi) = \frac{1}{2}\bigl((\vec\zeta\cdot\nabla)\vec u,~\vec\chi\bigr)-\frac{1}{2}\bigl((\vec\zeta\cdot\nabla)\vec\chi,~\vec u\bigr).
\end{equation*}

We next consider the material time derivative of the surface fluid velocity on the interface in \eqref{eq:if2}. 
We recall \eqref{eq:fdTd} to obtain
\begin{align}
&\ipd{\partial_t^\bullet \vec u, ~\vec\chi}_{\Gamma(t)}  = \ipd{\Dtv\vec u, ~\vec\chi}_{\Gt} + \ipd{([\vec u - \vec v]\cdot\nabs)\vec u,~\vec\chi}_{\Gt}\nn\\
&=\frac{1}{2}\ipd{\Dtv\vec u,~\vec\chi}_{\Gt}+\frac{1}{2}\ipd{\Dtv\vec u,~\vec\chi}_{\Gt} + \ipd{([\vec u - \vec v]\cdot\nabs)\vec u,~\vec\chi}_{\Gt},\nn\\
&=\frac{1}{2}\left[\ddt\ipd{\vec u,~\vec\chi}_{\Gt}+\ipd{\Dtv\vec u,~\vec\chi}_{\Gt}-\ipd{\vec u,~\Dtv\vec\chi}_{\Gt}\right]+\mathscr{A}_\Gt(\vec u-\vec v;\vec u,~\vec\chi),
\label{eq:fale2}
\end{align}
where we used \eqref{eq:fale1} from the appendix in the last equality and defined the skew-symmetric term
\begin{equation}
\mathscr{A}_{\Gt}(\vec\zeta;~\vec u, \vec\chi) = \frac{1}{2}\left[\ipd{(\vec\zeta\cdot\nabs)\vec u, ~\vec\chi}_{\Gt}-\ipd{(\vec\zeta\cdot\nabs)\vec\chi, ~\vec u}_{\Gt}\right].\label{eq:ifskew}
\end{equation}

Now, we combine \eqref{eq:ale2} and \eqref{eq:fale2} and propose the weak formulation for the bulk-surface Navier--Stokes equations as follows. For each $t\in(0,T]$, we find $\vec u\in\mathbb{V}_\Gamma(\vec g)$, $p\in\mathbb{P}$ and $p_\Gamma\in L^2(\Gt)$ such that
\begin{subequations}\label{eqn:nvweak}
\begin{align}
\label{eq:nvweak1}
&\frac{1}{2}\left[\ddt\bigl(\rho\,\vec u,~\vec\chi\bigr)+ \bigl(\rho\,\Dtw\vec u,~\vec\chi\bigr) - \bigl(\rho\,\vec u,~\Dtw\vec\chi\bigr)+\rho_+\ipd{\vec u\cdot\vec n_{_{\partial\Omega}},~\vec u\cdot\vec\chi}_{\partial_2\Omega}\right]\nn\\
&\qquad\quad + 2\bigl(\mu\mat{D}(\vec u),~\mat{D}(\vec\chi)\bigr)-\bigl(p,~\nabla\cdot\vec\chi\bigr) + \mathscr{A}\bigl(\rho[\vec u-\vec w];~\vec u,\vec\chi\bigr)\nn\\
&\qquad\quad +\frac{\rho_\Gamma}{2}\left[\ddt\ipd{\vec u,~\vec\chi}_{\Gt}+\ipd{\Dtv\vec u,~\vec\chi}_{\Gt}-\ipd{\vec u,~\Dtv\vec\chi}_{\Gt}\right]\nn\\
&\qquad\quad +2\mu_\Gamma\ipd{\mat{D}_s(\vec u),~\mat{D}_s(\vec\chi)}_{\Gt}-\ipd{p_\Gamma,~\nabs\cdot\vec\chi}_{\Gt}\nn\\
&\qquad\quad +\mathscr{A}_{\Gt}(\rho_\Gamma[\vec u-\vec v];~\vec u, \vec\chi)=  \alpha\ipd{f_\Gamma\,\vec\nu,~\vec\chi}_{\Gt}\qquad\forall\vec\chi\in \mathbb{V}_\Gamma(\vec 0),\\[0.6em]
\label{eq:nvweak2}
&\quad\bigl(\nabla\cdot\vec u,~q\bigr)=0\qquad\forall q\in\mathbb{P},\\[0.6em]
&\quad\ipd{\nabs\cdot\vec u,~q_\Gamma}_{\Gt}=0\qquad\forall q_\Gamma\in L^2(\Gt),
\label{eq:nvweak3}
\end{align}
\end{subequations}
where the interface velocity $\vec v$ and the curvature force $f_\Gamma$ will be computed in Section~\ref{sec:wffG}.

\subsection{Weak formulation of the curvature force}\label{sec:wffG}

The main idea to allow for an energy-stable discretization is based on the work in \cite{GNZ26willmore}. We note that
\begin{equation*}
\partial_t^\square \varkappa = \Delta_s\mathscr{V}+\mathscr{V}|\nabs\vec\nu|^2,\qquad\mbox{on}\quad\Gamma(t),
\end{equation*}
where $\mathscr{V}=\vec v\cdot\vec\nu$ is the velocity in the normal direction and $\partial_t^\square$ is the normal time derivative, which measures the change in the moving surface in the normal direction. Then, we incorporate the tangential velocity to obtain
\cite[Lemma 39]{Barrett20}
\begin{equation}
\label{eq:tvark}
\Dtv\varkappa = \Delta_s\mathscr{V} + \mathscr{V}|\nabs\vec\nu|^2+\vec v\cdot\nabs\varkappa\qquad\mbox{on}\quad\Gamma(t).
\end{equation}

Now, we derive a weak formulation for the curvature force $f_\Gamma$ in \eqref{eq:fGamma} using the curvature evolution equation \eqref{eq:tvark}. It follows from the transport theorem and the Gauss theorem on manifolds (see \cite[Theorems 21,32]{Barrett20}) that
\begin{align}
\ddt\ipd{\varkappa-\bkap,~\chi}_{\Gt}&=\ipd{\Dtv\varkappa,~\chi}_{\Gt}+ \ipd{\varkappa-\bkap,~\Dtv\chi}_{\Gt}+\ipd{(\varkappa-\bkap)\,\chi,~\nabs\cdot\vec v}_{\Gt}\nn\\
&=\ipd{\Dtv\varkappa,~\chi}_{\Gt}+ \ipd{\varkappa-\bkap,~\Dtv\chi}_{\Gt}\nn\\
&\quad -\ipd{[\varkappa-\bkap]\,\varkappa\,\chi,~\mathscr{V}}_{\Gt}- \ipd{\vec v,~\nabs[(\varkappa-\bkap)\,\chi]}_{\Gt}.\label{eq:ffale1}
\end{align}
Furthermore, we have
\begin{align}
&\ipd{\Dtv\varkappa - \vec v\cdot\nabs\varkappa,~\chi}_{\Gt}
=\frac{1}{2}\ipd{\Dtv\varkappa,~\chi}_{\Gt}+\frac{1}{2}\ipd{\Dtv\varkappa,~\chi}_{\Gt}- \ipd{(\vec v\cdot\nabs)[\varkappa-\bkap],~\chi}_{\Gt}\nn\\
&=\frac{1}{2}\left[\ddt\ipd{\varkappa-\bkap,~\chi}_{\Gt}+ \ipd{\Dtv\varkappa,~\chi}_{\Gt} - \ipd{\varkappa-\bkap,~\Dtv\chi}_{\Gt}\right]
\nn\\
&\qquad +\frac{1}{2}\ipd{[\varkappa-\bkap]\,\varkappa\,\chi,~\mathscr{V}}_{\Gt}-\mathscr{A}_{\Gamma(t)}(\vec v; \varkappa-\bkap, \chi),\label{eq:ffale2}
\end{align}
where we have used \eqref{eq:ffale1} in the last equality, and where $\mathscr{A}_{\Gt}$ is the skew-symmetric term 
\begin{equation*}
\mathscr{A}_{\Gt}(\vec\zeta;~\varkappa, \chi) = \frac{1}{2}\left[\ipd{(\vec\zeta\cdot\nabs)\varkappa, ~\chi}_{\Gt}-\ipd{(\vec\zeta\cdot\nabs)\chi, ~\varkappa}_{\Gt}\right],
\end{equation*}
defined in a similar manner to \eqref{eq:ifskew}.

On combining \eqref{eq:fGamma} and \eqref{eq:ffale1}, we now propose the following weak formulation for the computation of the curvature force $f_\Gamma$. For each $t\geq 0$, we find $f_\Gamma\in L^2(\Gt)$, $\varkappa\in H^1(\mathcal{G}_T)$ and $\mathscr{V}\in H^1(\Gt)$ such that
\begin{subequations}\label{eqn:ifweak}
\begin{align}\label{eq:ifweak1}
&\ipd{f_\Gamma,~\varphi}_{\Gt}-\ipd{\nabs\varkappa,~\nabs\varphi}_{\Gt} + \ipd{[\varkappa-\bkap]|\nabs\vec\nu|^2,~\varphi}_{\Gt}\nn\\
&\qquad\quad-\frac{1}{2}\ipd{[\varkappa-\bkap]^2\varkappa,~\varphi}_{\Gt}=0\qquad\forall\varphi\in H^1(\Gt),\\[0.5em]
\label{eq:ifweak2}
&\frac{1}{2}\left[\ddt\ipd{\varkappa-\bkap,~\chi}_{\Gt}+ \ipd{\Dtv\varkappa,~\chi}_{\Gt} - \ipd{\varkappa-\bkap,~\Dtv\chi}_{\Gt}\right]\nn\\
&\qquad\quad -\mathscr{A}_{\Gt}(\vec v;~\varkappa-\bkap,~\chi) +\ipd{\nabs\mathscr{V},~\nabs\chi}_{\Gt} -\ipd{\mathscr{V}|\nabs\vec\nu|^2,~\chi}_{\Gamma(t)}\nn\\
&\qquad\quad+\frac{1}{2}\ipd{\mathscr{V},~[\varkappa-\bkap]\,\varkappa\,\chi}_{\Gt}=0\qquad\forall\chi\in H^1(\mathcal{G}_T),\\[0.7em]
&\ipd{\mathscr{V},~\xi}_{\Gt} - \ipd{\vec u\cdot\vec\nu,~\xi}_{\Gt}=0\qquad\forall\xi\in L^2(\Gt),
\label{eq:ifweak3}
\end{align} 
\end{subequations}
where $\vec v$ is the interface velocity, and its computation will be presented in Section~\ref{sec:inftan}.

\subsection{Evolving interface with tangential velocity}\label{sec:inftan}

It remains to obtain a formulation for the interface velocity $\vec v$ that is needed to update the parameterization of $\Gt$ through \eqref{eq:para}. This can be done by supplementing the normal velocity $\mathscr{V}$ with a suitable tangential velocity. Indeed, similar to \cite{GNZ26willmore}, we are able to freely choose the tangential velocity. For example, we may choose a tangential velocity which under discretization leads to an improved surface mesh quality. Since the fluid velocity $\vec u$ is incompressible on the surface, directly updating the interface with this velocity already leads to a very good mesh. This gives the following weak formulation: for each $t\in(0,T]$, we find $\vec x\in H^1(\Upsilon)$ with $\vec v\in [H^1(\Gt)]^d$ and $\lambda_\Gamma\in L^2(\Gt)$ such that
\begin{subequations}\label{eq:Xweak}
\begin{align}\label{eq:Xweak1}
&\ipd{\vec\nu\cdot\vec v,~\phi}_{\Gt} = \ipd{\mathscr{V},~\phi}_{\Gt}  \qquad\forall\phi\in L^2(\Gt),\\
&\ipd{\mat{P}_\Gamma\vec v,~\vec\eta}_{\Gt} + \ipd{\lambda_\Gamma\,\vec\nu, ~\vec\eta}_{\Gt}=  \ipd{\mat{P}_\Gamma\vec u,~\vec\eta}_{\Gt}\qquad\forall\vec\eta\in [H^1(\Gt)]^d,\label{eq:Xweak2}
\end{align}
\end{subequations}
where for convenience we have introduced the trivial Lagrange multiplier $\lambda_\Gamma = 0$ that will facilitate the introduction of compact well-posed fully discrete schemes upon discretization later on.

Now, the weak formulation for the whole system is given as follows. Given the initial surface $\Gamma(0) = \Gamma_0$ and $\varkappa(\cdot,0) = \varkappa_0$ and fluid velocity $\vec u_0$, for each $t\in(0,T]$, we find $\Gt$ with $\vec v\in [H^1(\Gt)]^d$, $(\vec u,~p,~p_\Gamma)\in\mathbb{V}_\Gamma(\vec g)\times\mathbb{P}\times L^2(\Gt)$, and $(f_\Gamma, \varkappa,~\mathscr{V})\in L^2(\Gt)\times H^1(\mathcal{G}_T)\times H^1(\Gt)$ by solving \eqref{eqn:nvweak}, \eqref{eqn:ifweak} and \eqref{eq:Xweak}. Here, we note that $\mathscr{A}_{\Gamma(t)}(\rho_\Gamma[\vec u-\vec v];\vec u,\vec\chi)$ in \eqref{eq:nvweak1} vanishes since the fluid tangential velocity is incorporated for the evolving interface, recalling \eqref{eq:Xweak}.

The following theorem is the analogue of Lemma~\ref{lem:dtE}, which we now
prove in the context of the weak formulation.
\begin{thm}\label{thm:weak} If $\vec g = \vec 0$, then the weak solution
satisfies the energy law
\begin{subequations}
\begin{align}
\label{eq:energylawweak}
&\frac{\rm d}{\rm d t} E(t) 
+ 2\bigl(\mu\mat{D}(\vec u),~\mat{D}(\vec u)\bigr) + 2\mu_\Gamma\ipd{\mat{D}_s(\vec u),~\mat{D}_s(\vec u)}_{\Gt}=-\frac{\rho_+}{2}\ipd{|\vec u|^2, \vec u\cdot\vec n_{_{\partial\Omega}}}_{\partial_2\Omega}.
\end{align}
Besides, the weak solution satisfies 
\begin{align}\label{eq:volumeweak}
\ddt\vol(\Gamma(t)) &= 0.\qquad t\in[0,T], \\
\ddt|\Gamma(t)| &= 0.\qquad t\in[0,T]. \label{eq:areaweak}
\end{align}
\end{subequations}
\end{thm}
\begin{proof}
Choosing $\vec\chi=\vec u$ in \eqref{eq:nvweak1}, $q=p$ in \eqref{eq:nvweak2}, $q_\Gamma=p_\Gamma$ in \eqref{eq:nvweak3}, $\varphi = \mathscr{V}$ in \eqref{eq:ifweak1}, $\chi=(\varkappa-\bkap)$ in \eqref{eq:ifweak2} and $\xi=f_\Gamma$ in \eqref{eq:ifweak3}, and then combining these equations, we obtain \eqref{eq:energylawweak} straightforwardly. 

It follows from the Reynolds transport theorem that
\begin{equation}
\ddt\vol(\Gamma(t)) = 
\ipd{\vec\nu\cdot\vec v,~1}_{\Gt}.\label{eq:volumew1}
\end{equation}
Now, we choose $q = \mX_{\Omega_-(t)}-\lambda(t)$ in \eqref{eq:nvweak2} with $\lambda(t)=\frac{\vol(\Omega_-(t))}{\vol(\Omega)}$, $\xi=1$ in \eqref{eq:ifweak3} and $\phi=1$ in \eqref{eq:Xweak1} and combine these equations. Then, we obtain
\begin{equation}
\ipd{\vec\nu\cdot\vec v,~1}_{\Gt} = \ipd{\mathscr{V},~1}_{\Gt} = \ipd{\vec u\cdot\vec\nu,~1}_{\Gt} = \ipD{\nabla\cdot\vec u,~\mX_{_{\Omega_-(t)}}} =0,
\end{equation}
which implies \eqref{eq:volumeweak}, on recalling \eqref{eq:volumew1}.
Finally, it follows from \cite[Theorems~32, 21]{Barrett20}, \eqref{eq:Xweak1}
and \eqref{eq:ifweak3} that
\begin{align*}
\ddt|\Gamma(t)| = \ipd{\nabs\cdot\vec v,~1}_{\Gt} = -\ipd{\varkappa, \vec v \cdot \vec\nu}_{\Gt}
= -\ipd{\varkappa, \vec u \cdot \vec\nu}_{\Gt}= \ipd{\nabs\cdot\vec u,~1}_{\Gt} = 0,
\end{align*}
which yields \eqref{eq:areaweak}. 
\end{proof}

\begin{rem}\label{rem:tanwk}
It is noteworthy that the tangential velocity can be chosen freely without affecting the results in Theorem~\ref{thm:weak} for the weak solution. For instance, one could consider the `BGN' tangential velocity by replacing \eqref{eq:Xweak2} with  \cite{Barrett20}
\begin{equation}
\ipd{\lambda_\Gamma\,\vec\nu,~\vec\eta}_{\Gt}+\ipd{\nabs\vec\id,~\nabs\vec\eta}_{\Gt}=0\qquad\forall\vec\eta\in[H^1(\Gt)]^d,
\end{equation}
with $\lambda_\Gamma = \varkappa$ in the continuous setting. An alternative approach that could lead to good mesh quality is the minimal deformation rate (MDR) approach by replacing \eqref{eq:Xweak2} with \cite{Hu22evolving}
\begin{equation}
\ipd{\lambda_\Gamma\,\vec\nu,~\vec\eta}_{\Gt}+\ipd{\nabs\vec v,~\nabs\vec\eta}_{\Gt}=0\qquad\forall\vec\eta\in[H^1(\Gt)]^d.
\end{equation}
In the present work, for simplicity, we update the interface directly according to the fluid velocity, which generally preserves the quality of the interface mesh due to the surface incompressibility condition \eqref{eq:if3}.
\end{rem}

\section{Finite element approximations}
\label{sec:FEM}

\subsection{The discretization}
We partition the time interval as
\[[0,T]=\cup_{m=1}^M [t_{m-1},~t_m]\qquad\mbox{with}\quad t_m = m\ttau,\quad \ttau = T/M,\]
with a uniform time step size $\ttau$.

\vspace{0.2cm}
\noindent{\bf Interface discretization.}\par\noindent
The hypersurface $\Gamma(t_m)$ is approximated by a $(d-1)$-dimensional polyhedral surface
\begin{equation}
 \Gamma^m:= \cup_{j=1}^{J_\Gamma}\overline{\sigma_j^m}\qquad\mbox{with}\quad Q_\Gamma^m=\{\vec q_k^m\}_{k=1}^{K_\Gamma},\label{eq:ifdisc}
 \end{equation}
where $\{\sigma_j^m\}_{j=1}^{J_\Gamma}$ are mutually disjoint $(d-1)$-simplices in $\bR^d$, and $Q_\Gamma^m$ is a collection of the vertices. On the polyhedral surface $\Gamma^m$, we define the finite element spaces
\begin{subequations}
\begin{align*}
S_k^h(\Gm) &:= \bigl\{\chi\in C(\Gm):\quad \chi|_{\sigma_j^m}\in P_k(\sigma_j^m)\quad  j =1,\ldots, J_\Gamma\bigr\}, \quad k\in\bN_+,\\
S_0^h(\Gm) &:= \bigl\{\chi\in L^2(\Gm):\quad \chi|_{\sigma_j^m}\;\;\mbox{is constant}\quad j=1,\ldots, J_\Gamma\bigr\},
\end{align*}
\end{subequations}
where $P_k(o)$ denotes the space of polynomials of degree at most $k$ on $\sigma_j^m$. We also let $V^h(\Gm) = S_1^h(\Gm)$.

In addition, let
$\left\{\vec q_{j_k}^{m}\right\}_{k=0}^{d-1}$ be the vertices of $\sigma_j^{m}$, ordered with the same orientation for all $\sigma_j^{m}$, $j=1,\ldots, J_\Gamma$. We also denote
$\sigma_j^{m}=\Delta\left\{\vec q_{j_k}^{m}\right\}_{k=0}^{d-1}$ and introduce the unit normal vector $\vec{\nu}^m$ to $\Gamma^m$, such that $\vec\nu^m\in [S_0^h(\Gm)]^d$
\begin{equation}\label{eq:vG}
\vec{\nu}^m_{j} := \vec{\nu}^m \mid_{\sigma^{m}_j} :=
\frac{\vec N\{\sigma_j^{m}\}}{
|\vec N\{\sigma_j^{m}\}|}\quad\mbox{ with}\quad \vec N\{\sigma_j^{m}\}=( \vec{q}^{m}_{j_1} - \vec{q}^{m}_{j_0} ) \wedge \ldots \wedge
( \vec{q}^{m}_{j_{d-1}} - \vec{q}^{m}_{j_0}),
\end{equation}
where $\wedge$ denotes the wedge product and $\vec N\{\sigma_j^{m}\}$ is the orientation vector of $\sigma_j^{m}$. To approximate the inner product $\langle\cdot,\cdot\rangle_{\Gamma(t_m)}$, we also introduce
$\langle\cdot,\cdot\rangle_{\Gamma^m}$ and $\langle\cdot,\cdot\rangle_{\Gamma^m}^h$ over the current polyhedral surface $\Gamma^m$ via
\begin{subequations}
\begin{align}
\label{eq:erule}
\langle a,~ b\rangle_{\Gamma^m} &:=
\int_{\Gamma^m} a\cdot b\; \dH^{d-1},\\
\langle a, b \rangle^h_{\Gamma^m} &:=
\frac{1}{d}\sum_{j=1}^{J_\Gamma} |\sigma^{m}_j|
\sum_{k=0}^{d-1}
\underset{\sigma^{m}_j\ni \vec{p}\to \vec{q}^{m}_{j_k}}{\lim}(a\cdot b)(
\vec{p}),
\label{eq:tprule}
\end{align}
\end{subequations}
where $a,b$ are piecewise continuous, with possible jumps only across the edges of $\{\sigma^{m}_j\}_{j=1}^{J_\Gamma}$, and
$|\sigma^{m}_j| = \frac{1}{(d-1)!}|\vec N\{\sigma_j^{m}\}|$
is the measure of $\sigma^{m}_j$. We also induce the discrete projection operator
\[\mat{P}_{\Gm} = \mat{\Id} - \vec\nu^m\otimes\vec\nu^m\qquad\mbox{on}\quad\Gm,\]
and the discrete surface deformation tensor
\[\mat{D}_s^m(\vec \eta) = \frac{1}{2}\mat{P}_{\Gamma^m}(\nabs\vec\eta + (\nabs\vec \eta)^T)\mat{P}_{\Gamma^m}\qquad\mbox{on}\qquad\Gm,\]
where $\nabs$ is taken on $\Gamma^m$.

The discrete normals $\vec\nu^m$ are piecewise constants. For later use, we follow the work in \cite{BGN08parametric} and introduce the vertex normal vector $\vec\nu_p^{m}\in [V^h(\Gm)]^d$.
Specifically, $\vec\nu_p^m\in [V^h(\Gm)]^d$ is defined as the mass-lumped $L^2$--projection of $\vec\nu^m$ onto $[V^h(\Gm)]^d$, i.e.,
\begin{equation} \label{eq:nuhomegah}
\ipd{\vec\nu_p^m, \vec\eta^h}_{\Gm}^h = \ipd{\vec\nu^m,~\vec\eta^h}_{\Gm}^h = \ipd{\vec\nu^m,~\vec\eta^h}_{\Gm}\quad\forall\vec\eta^h\in [V^h(\Gm)]^d.
\end{equation}
For the parametric finite element spaces associated with $\Gm$, we introduce the standard interpolation operator \[\vec\pi_l^m: [C(\Gm)]^d\to [S_l^h(\Gm)]^d.\]

\vspace{0.2cm}
\noindent{\bf Bulk discretization}: At time $t_m$, we consider a regular partition of $\Omega$ with $K^m_\Omega$ vertices as
\begin{equation*}
\overline{\Omega} = \cup_{j = 1}^{J_\Omega} \overline{o_j^m}\quad\mbox{with}\quad Q^m=\{\vec a_k^m\}_{k=1}^{K_\Omega},\quad \mathscr{T}^m=\bigl\{o_j^m: j = 1,\ldots,J_\Omega\bigr\},
\end{equation*}
where $\{o_j^m\}_{j=1}^{J_\Omega}$ are the mutually disjoint open simplices in $\bR^d$, and $\mathscr{T}^m$ is the bulk mesh. We then introduce the finite element spaces associated with $\mathscr{T}^m$ as
\begin{subequations}
\begin{align*}
S_k^m(\Omega) &:= \bigl\{\chi\in C(\overline{\Omega}):\quad \chi|_{o}\in P_k(o)\quad\forall o\in\mathscr{T}^m\bigr\}, \quad k\in\bN_+,\\
S_0^m(\Omega) &:= \bigl\{\chi\in L^2(\Omega):\quad \chi|_{o}\;\;\mbox{is constant}\quad\forall o\in\mathscr{T}^m\bigr\},
\end{align*}
\end{subequations}
where $P_k(o)$ denotes the space of polynomials of degree at most $k$ on $o$.

It is natural to consider the following discrete analogues of the (affine) function spaces in \eqref{eqn:UPspaces} with Taylor--Hood pair elements:
\begin{subequations}\label{eq:UPspaces}
\begin{align}\label{eq:P2P1}
\mbox{P2-P1}:&\quad\bigl(\mathbb{U}^m(\vec g),~\mathbb{P}^m\bigr)= \bigl([S_2^m(\Omega)]^d\cap\mathbb{U}(\vec g),~S_1^m(\Omega)\cap\mathbb{P}\bigr),\\\label{eq:P2P0}
\mbox{P2-P0}:&\quad\bigl(\mathbb{U}^m(\vec g),~\mathbb{P}^m\bigr)= \bigl([S_2^m(\Omega)]^d\cap\mathbb{U}(\vec g),~S_0^m(\Omega)\cap\mathbb{P}\bigr),\\\label{eq:P2P1P0}
\mbox{P2-(P1+P0)}:&\quad\bigl(\mathbb{U}^m(\vec g),~\mathbb{P}^m\bigr)= \bigl([S_2^m(\Omega)]^d\cap\mathbb{U}(\vec g),~(S_1^m(\Omega)+S_0^m(\Omega))\cap\mathbb{P}\bigr),
\end{align}
\end{subequations}
which usually satisfy the LBB inf-sup stability condition, see \cite{BrezziF91, Boffi2012local, BGN15stable}.

\subsection{The discrete ALE mappings}
\label{sec:dALEmap}

\vspace{0.2cm}
\noindent{\bf Mappings on the interface}: At each time step, we now assume that we are given the polyhedral surface
$\Gamma^{m}=\vec X^m(\Gamma^{m-1})$ with discretization given in \eqref{eq:ifdisc}. We introduce the discrete surface velocity $ \mathscr{\vv V}^m\in [V^h(\Gm)]^d$ defined by
\[\mathscr{\vv V}^m(\vec q_k^m) = \frac{\vec q_k^m - \vec q_k^{m-1}}{\ttau},\quad\mbox{for}\quad k = 1,\ldots, K_\Gamma.\]
We also define the continuous-in-time ALE mapping
\begin{equation}\label{eq:ALESur}
\vec\Phi_\Gamma^m[t]=\vec\id|_{\Gm} - \mathscr{\vv V}^m\,(t_m-t),\qquad t\in[t_{m-1}, t_m],
\end{equation}
which provides a mapping between the physical domain at time $t$ and the surface reference configuration $\Gamma^m$. In particular, by \eqref{eq:fALET}, we note that it holds that
\begin{equation}
\Dtv(\vec\eta\circ\vec\Phi_\Gamma^m[t])=\vec 0,\qquad\forall\vec\eta\in [H^1(\Gamma^m)]^d.
\end{equation}
The mappings in \eqref{eq:ALESur} lead to discrete mappings $\vec\Phi_\Gamma^m[t_{m-1}]\in V^h(\Gm)$ such that
\begin{equation*}
\vec\Phi_\Gamma^m[t_{m-1}] = \vec\id|_{\Gm} - \mathscr{\vv V}^m\,\ttau\quad\mbox{with}\quad \vec\Phi_\Gamma^m[t_{m-1}](\Gm) = \Gamma^{m-1}, \quad \vec\Phi_\Gamma^m[t_{m-1}](\vec q_k^m) = \vec q_k^{m-1},
\end{equation*}
for $1\leq k\leq K_\Gamma$. This also yields the discrete pushforward and pullback operators as follows:
\begin{subequations}\label{eqn:surfmap}
\begin{align}
&\Pi_{m-1}^{m}: C(\Gamma^{m-1})\to C(\Gm)\quad\mbox{so that}\quad \Pi_{m-1}^m z= z\circ\vec\Phi_\Gamma^m[t_{m-1}]\quad\forall z\in C(\Gamma^{m-1});\\
&\Pi_m^{m-1}: C(\Gm)\to C(\Gamma^{m-1})\quad\mbox{so that}\quad \Pi_{m}^{m-1} z = z\circ\{\vec\Phi_\Gamma^m[t_{m-1}]\}^{-1}\quad\forall z\in C(\Gamma^{m});
\end{align}
\end{subequations}
with the above definitions naturally extended to vector-valued function spaces.

\vspace{0.2cm}
\noindent{\bf Mappings in the bulk}: We first consider the approximation in the framework of the fitted mesh approach. At each time step, we again assume that we are given the polyhedral surface
$\Gamma^{m}=\vec X^m(\Gamma^{m-1})$.
Based on the previous bulk mesh $\mathscr{T}^{m-1}$, we then construct the new bulk mesh $\mathscr{T}^{m}$ by updating the vertices of the mesh according to
\begin{equation}
\vec a_k^{m} = \vec a_k^{m-1} + \vec\psi^m(\vec a_k^{m-1}),\qquad 1\leq k\leq K_\Omega, \quad 1\leq m\leq M.\label{eq:bdd}
\end{equation}
Here we recall that $Q^m=\{\vec a_k^m\}_{k=1}^{K_\Omega}$ are the vertices of $\mathscr{T}^m$, and $\vec\psi^m\in [S_1^{m-1}(\Omega)]^d$ is the displacement of the bulk mesh, which is obtained by seeking
$\vec\psi^m\in\mathbb{Y}^{m-1}$ such that
\begin{equation}
2\bigl(\lambda^{m-1}\,\mat{\tD}(\vec\psi^m),~\mat{\tD}(\vec\chi)\bigr) + \bigl(\lambda^{m-1}\,\nabla\cdot\vec\psi^m,~\nabla\cdot\vec\chi\bigr)=0\qquad\forall\vec\chi\in\mathbb{Y}_0^{m-1}, \label{eq:elastic}
\end{equation}
where we define
\begin{align*}
\mathbb{Y}^{m-1}&=\bigl\{\vec\chi\in[S_1^{m-1}(\Omega)]^d:\,\vec\chi\cdot\vec n _{\partial\Omega}= 0\;\;\mbox{on}\;\;\partial\Omega;\;\vec\chi = \vec X^m - \vec\id\;\;\mbox{on}\;\;\Gamma^{m-1}\bigr\},\\
\mathbb{Y}_0^{m-1}&=\bigl\{\vec\chi\in[S_1^{m-1}(\Omega)]^d:\,\vec\chi\cdot\vec n_{\partial\Omega} = 0\;\;\mbox{on}\;\;\partial\Omega;\;\vec\chi = \vec 0\;\;\mbox{on}\;\;\Gamma^{m-1}\bigr\},\\
\lambda^{m-1} &\in S^{m-1}_0(\Omega),\qquad (\lambda^{m-1})_{|o} = 1 + \frac{\max\limits_{o\in\mathscr{T}^{m-1}} |o| - \min\limits_{o\in\mathscr{T}^{m-1}}|o|}{|o|},
\qquad \forall o\in\mathscr{T}^{m-1}.
\end{align*}
Here, $\lambda^m$ is used to limit the distortion of small elements \cite{Masud1997space,GNZ23}.

In the aforementioned mesh moving strategy, the mesh velocity $\vec W^m(\cdot)\in [S_1^m(\Omega)\cap H_0^1(\Omega)]^d$ is defined by
\begin{equation*}
\vec W^m(\vec a_k^m) = \frac{\vec a_k^{m}-\vec a_k^{m-1}}{\ttau},\quad\mbox{for}\quad k = 1,\ldots, K_\Omega,
\end{equation*}
which also satisfies the constraint that
\begin{equation}\label{eq:ALEcond}
 \vec W^m =  \mathscr{\vv V}^m\quad\mbox{on}\quad\Gamma^m,
 \end{equation}
to ensure that the interface mesh $\Gamma^m$ remains fitted to the bulk mesh $\mathscr{T}^m$.

Similarly, we introduce the continuous-in-time ALE mapping
\begin{equation}\label{eq:ALEBulk}
\vec\Phi^m[t] = \vec\id|_{\Omega}- \vec W^m\,(t_m-t),\quad t\in[t_{m-1}, t_m],
\end{equation}
which provides a mapping between the physical domain at time $t$ and the reference configuration. In particular, by \eqref{eq:ALET} we have
\begin{equation}
\Dtw (\vec\eta\circ\vec\Phi^m[t]) = 0\qquad\forall\vec\eta\in[S_k^m(\Omega)]^d.
\end{equation}
The mapping \eqref{eq:ALEBulk} at $t_{m-1}$ also yields a discrete mapping
\begin{equation*}
\vec\Phi^m[t_{m-1}]=\vec\id|_{\Omega} - \vec W^m\,\ttau,\quad\mbox{with}\quad \vec\Phi^m[t_{m-1}](\Omega) = \Omega, \quad \vec\Phi^m[t_{m-1}](\vec a_k^m) = \vec a_k^{m-1},
\end{equation*}
for $1\leq k\leq K_\Omega$. Similarly, we also define the pushforward and pullback operators on $\Omega$ as
\begin{subequations}
 \begin{align*}
&\Lambda_{m-1}^{m}: C(\bar{\Omega})\to C(\bar{\Omega})\quad\mbox{so that}\quad \Lambda_{m-1}^m z= z\circ\vec\Phi^m[t_{m-1}]\quad\forall z\in C(\bar{\Omega});\\
&\Lambda_m^{m-1}: C(\bar{\Omega})\to C(\bar{\Omega})\quad\mbox{so that}\quad \Lambda_{m}^{m-1} z = z\circ(\vec\Phi^m[t_{m-1}])^{-1}\quad\forall z\in C(\bar{\Omega}), 
\end{align*}
\end{subequations}
with the above definitions naturally extended to vector-valued function spaces.

\subsection{The ALE method}

In the fitted mesh approach, the bulk mesh $\mathscr{T}^m$ can be simply divided into the interior and exterior elements $\mathscr{T}_-^m$ and $\mathscr{T}_+^m$
\begin{equation*}
 \mathscr{T}_{\pm}^m:=\left\{o\in\mathscr{T}^m:\; o\subset\Omega_\pm^m\right\},
 \end{equation*}
where $\Omega_-^m$ and $\Omega_+^m$ again denote the interior and exterior of $\Gamma^m$, respectively. Therefore, the discrete densities and viscosities can be defined naturally as
\begin{equation}
\rho^m = \rho_-\mX_{_{\Omega_-^m}}+\rho_+\mX_{_{\Omega_+^m}},\qquad \mu^m = \mu_-\mX_{_{\Omega_-^m}}+\mu_+\mX_{_{\Omega_+^m}}.\label{eq:DisQ}
\end{equation}

We are now ready to introduce the fully discrete ALE finite element method for the coupled system. We employ the $\mbox{P2-(P1+P0)}$ element pair, as defined in \eqref{eq:P2P1P0}, and the scheme is split into three parts, which provide approximations for the Navier--Stokes equations, the curvature force, and the evolving surface.

\noindent{\bf For the NS equations}:
Let $\Gamma^0$ be an approximation of $\Gamma(0)$ and $\vec U^0\in \bU^0(\vec g)$ be the initial discrete velocity. We set $\vec W^0=\vec 0$. For each $m\geq 0$, we find $(\vec U^{m+1}, P^{m+1}, P_\Gamma^{m+1})\in \bU^m(\vec g)\times\bP^m\times V^h(\Gm)$ and $P_{\rm sing}^{m+1}\in\bR$ such that
\begin{subequations}\label{eqn:fdnv}
\begin{align}
\label{eq:fdnv1}
&\frac{1}{2\,\ttau}\left[\ipD{\rho^m\,\vec U^{m+1}, ~\vec\xi^h} + \ipD{\rho^{m-1}(\Lambda_{m}^{m-1}\vec U^{m+1}), ~\Lambda_m^{m-1}\vec\xi^h}\right]-\ipD{P^{m+1},~\nabla\cdot\vec\xi^h}\nn\\
&\quad+ 2\ipD{\mu^m\,\mat{D}(\vec U^{m+1}),~\mat{D}(\vec\xi^h)}-P_{\rm sing}^{m+1}\ipd{\vec\nu_p^m,~\vec\xi^h}_{\Gm}^h +\mathscr{A}\left(\rho^m[\Lambda_{m-1}^m\vec U^m-\vec W^{m}];\vec U^{m+1},\vec\xi^h\right) \nn\\
&\quad+\frac{\rho_\Gamma}{2\,\ttau}\left[\ipd{\vec U^{m+1},~\vec\xi^h}_{\Gm}+ \ipd{\Pi_m^{m-1}\vec U^{m+1},~\Pi_m^{m-1}\vec\xi^h}_{\Gamma^{m-1}}\right] \nn\\
&\quad- \ipd{P^{m+1}_\Gamma,~\nabs\cdot[\vec\pi_1^m\vec\xi^h]}_{\Gm}   +2\mu_\Gamma\ipd{\mat{D}_s^m(\vec U^{m+1}),~\mat{D}_s^m(\vec\xi^h)}_{\Gm}\nn\\
&\quad = \frac{1}{\ttau}\ipD{\rho^{m-1}\,\vec U^m,~\vec\xi^h} + \frac{\rho_\Gamma}{\ttau}\ipd{\vec U^m,~\Pi_m^{m-1}\vec\xi^h}_{\Gamma^{m-1}}+ \alpha\ipd{F_\Gamma^{m+1}\vec\nu^m,~\vec\xi^h}_{\Gm}^h\nn\\
&\quad\qquad  -\frac{1}{2}\ipd{\rho_+\vec U^m\cdot\vec n_{\partial\Omega},~\vec U^m\cdot\vec\xi^h\,}_{\partial_2\Omega}\qquad\forall\vec\xi^h\in\bU^m(\vec 0); \\[0.5em]
\label{eq:fdnv2}
&\ipD{\nabla\cdot\vec U^{m+1},~q^h} = 0\qquad\forall q^h\in \bP^m,\qquad\mbox{and}\quad \ipd{\vec U^{m+1}, \vec\nu_p^m}_{\Gm}^h=0;\\[0.5em]
&\ipd{\nabs\cdot[\vec\pi_1^m\vec U^{m+1}],~q^h_{\Gamma}}_{\Gamma^m}=0\qquad\forall q^h_\Gamma\in V^h(\Gm);
\label{eq:fdnv3}
\end{align}
\end{subequations}
where the computation of $F_\Gamma^{m+1}$ will be given later. Here $P_{\rm sing}^{m+1}$ is the Lagrange multiplier for the constraint 
\begin{equation*}
\quad \ipd{\vec U^{m+1}, \vec\nu_p^m}_{\Gm}^h=0,
\end{equation*}
which helps improve the volume preservation of the scheme.

\noindent{\bf For the curvature equations}:
Let $\varkappa^0\in V^h(\Gamma^0)$ be an approximation of the initial curvature $\varkappa(\cdot,0)$, and set $\mathscr{\vv V}^{0}=\vec 0$. For each $m\geq0$, we first obtain an appropriate approximation $\mathcal{W}^m \in L^\infty(\Gamma^m)$ of $|\nabs\vec\nu|^2$ at $t_m$. Specifically, we let
\begin{equation*} 
\mathcal{W}^m =| \nabs\vec\nu_n^m|^2\quad\mbox{with}\quad  \vec\nu_n^m(\vec q) =\frac{\vec\nu_p^m(\vec q)}{|\vec\nu_p^m(\vec q)|}\qquad\forall\vec q\in Q_\Gamma^m, 
\end{equation*}
where $\vec\nu_p^m$ is the vertex normal defined in \eqref{eq:nuhomegah}, which from now on we assume to be nondegenerate, and $\vec \nu_n^m\in[V^h(\Gm)]^d$ is the normalized vertex normal. Next, we find $(F_\Gamma^{m+1}, \varkappa^{m+1}, \mathscr{V}^{m+1})\in [V^h(\Gamma^m)]^3$ such that
\begin{subequations}\label{eqn:fdsf}
\begin{align}
\label{eq:fdsf1}
&\ipd{F_\Gamma^{m+1},~\varphi^h}_{{\Gm}}^h-\ipd{\nabs\varkappa^{m+1},~\nabs\varphi^h}_{\Gm} + \ipd{\mathcal{W}^m\,[\varkappa^{m+1}-\bkap],~\varphi^h}_{\Gm}\nn\\
&\qquad\quad -\frac{1}{2}\ipd{[\Pi_{m-1}^m\varkappa^m-\bkap]\,\Pi_{m-1}^m\varkappa^m\,[\varkappa^{m+1}-\bkap],~\varphi^h}_{\Gm}=0\quad\forall\varphi^h\in V^h(\Gm);\\[0.5em]
\label{eq:fdsf2}
&\frac{1}{2\,\ttau}\Bigl[\ipd{\varkappa^{m+1}-\bkap, ~\chi^h}_{\Gamma^m} + \ipd{\Pi_{m}^{m-1}\varkappa^{m+1}-\bkap,~\Pi_{m}^{m-1}\chi^h}_{\Gamma^{m-1}}\Bigr]\nn\\
&\qquad\quad-\mathscr{A}_{\Gamma^m}(\mathscr{\vv V}^{m}; \varkappa^{m+1}-\bkap,\chi^h)+\ipd{\nabs\mathscr{V}^{m+1}, \nabs\chi^h}_{\Gamma^m}-\ipd{\mathscr{ V}^{m+1}\,\mathcal{W}^m,~\chi^h}_{\Gamma^m}\nn\\
&\qquad\qquad  + \frac{1}{2}\ipd{\mathscr{V}^{m+1},~[\Pi_{m-1}^{m}\varkappa^m-\bkap]\Pi_{m-1}^{m}\varkappa^m\,\chi^h}_{\Gamma^m}\nn\\
&\qquad\qquad = \frac{1}{\ttau}\ipd{\varkappa^m-\bkap,~\Pi_{m}^{m-1}\chi^h}_{\Gamma^{m-1}}\qquad\forall\chi^h\in V^h(\Gamma^m);\\[0.5em]
&\ipd{\mathscr{V}^{m+1},~\zeta^h}_{\Gm}^h-\ipd{\vec U^{m+1}\cdot\vec\nu^m,~\zeta^h}_{\Gm}^h=0\qquad\forall\zeta^h\in V^h(\Gm);
\label{eq:fdsf3}
\end{align}
\end{subequations}
where $\mathscr{V}^{m+1}$ will be used to determine the new interface.

\noindent{\bf For the evolving surface}:
In our parametric finite element approximations, we have the freedom to choose the tangential velocity to enhance the mesh quality. Since the surface fluid velocity is incompressible, we choose to evolve the interface directly using the fluid velocity. This provides the finite element approximations of the weak formulation in \eqref{eq:Xweak}. For each $m\geq 0$ and given $\mathscr{V}^{m+1}\in V^h(\Gm)$, we seek $\vec X^{m+1}\in [V^h(\Gm)]^d$ and $\lambda^{m+1}_\Gamma\in V^h(\Gm)$ such that
\begin{subequations}\label{eqn:Xfull}
 \begin{align}
 \label{eq:Xfull1}
&\ipd{\frac{\vec X^{m+1}-\vec\id}{\ttau}\cdot\vec\nu_p^m,~\eta^h}_{\Gm}^h = \ipd{\mathscr{V}^{m+1},~\eta^h}_{\Gm}^h\qquad\forall\eta^h\in V^h(\Gm),\\
\label{eq:Xfull2}
&\ipd{\mat{P}_{\Gm}\frac{\vec X^{m+1}-\vec\id}{\ttau},~\vec\eta^h}_{\Gm}^h +\ipd{\lambda^{m+1}_\Gamma\,\vec\nu_p^m,~\vec\eta^h}_{\Gm}^h
= \ipd{\mat{P}_{\Gm}\vec U^{m+1},~\vec\eta^h}_{\Gm}^h\qquad\forall\vec\eta^h\in [V^h(\Gm)]^d.
 \end{align}
\end{subequations}
Once $\vec X^{m+1}$ is obtained, it defines the new polyhedral surface through $\Gamma^{m+1} = \vec X^{m+1}(\Gm)$.

In practice, at each time step, we have two stages. In the first stage, we solve for \eqref{eqn:fdnv} and \eqref{eqn:fdsf} simultaneously to obtain $\vec U^{m+1}$, $\varkappa^{m+1}$ and $\mathscr{V}^{m+1}$. At the next stage, we then
use $(\mathscr{V}^{m+1}, \vec U^{m+1})$ to update the fluid interface $\Gamma^{m+1}$ according to \eqref{eqn:Xfull}. This then provides the required data to perform our moving mesh strategies in \eqref{eq:bdd} for the next time step.

\begin{rem}
For the approximation of the surface incompressibility, the interpolation operator $\vec\pi_1^m$ is employed in both \eqref{eq:fdnv1} and \eqref{eq:fdnv3} in order to improve the surface area preservation of the scheme, see also the unfitted approach in \cite{BGN16sfluidic}. 
\end{rem}

\subsection{Properties of the method}\label{sec:pro}

We now investigate the solvability of the introduced method. We denote by $LBB_\Gamma$ the  inf-sup stability condition if there exists a $C_0>0$ that is independent of $\mathscr{T}^h$ and $\{\sigma_j^m\}_{j=1}^{J_\Gamma}$ such that
\begin{equation}
 \inf_{(\varphi, \eta,\lambda)\in \mathbb{P}\times V^h(\Gm)\times\bR}\sup_{\vec\xi\in \bU^m(\vec 0)}\frac{\ipD{\varphi,~\nabla\cdot\vec\xi} + \ipd{\eta,~\nabla_s\cdot[\vec\pi_1^m\vec\xi]}_{\Gm} + \lambda\ipd{\vec\nu_p^m, \vec\xi}^h_{\Gm}}{(\norm{\varphi}_0 + \norm{\eta}_{\Gm,0} + |\lambda|)(\norm{\vec\xi}_1 + \norm{\mat{P}_{\Gm}[\vec\pi_1^m\vec\xi]}_{\Gm,1})}\geq C_0>0,\label{eq:infsup}
 \end{equation}
where $\norm{\cdot}_0$ and $\norm{\cdot}_{1}$ are the $L^2$-- and $H^1$--norm on $\Omega$, and $\norm{\cdot}_{\Gm,0}$ and $\norm{\cdot}_{\Gm,1}$ are the $L^2$-- and $H^1$--norm on $\Gm$ with $\norm{\vec\xi}_{\Gm,1}^2:=\ipd{\vec\xi,~\vec\xi}_{\Gm} + \ipd{\nabla_s\vec\xi, ~\nabla_s\vec\xi}_{\Gm}$. Compare with
\cite[(5.2)]{BGN17fluidic}.

We have the following theorem for the coupled linear system.
\begin{thm}[existence and uniqueness]
Assuming that the $LBB_\Gamma$ inf-sup stability condition \eqref{eq:infsup} holds, then the linear system of \eqref{eqn:fdnv}, \eqref{eqn:fdsf} admits a unique solution $(\vec U^{m+1}, P^{m+1}, P_\Gamma^{m+1}, F_\Gamma^{m+1}, \varkappa^{m+1}, \mathscr{V}^{m+1}, P_{\rm sing}^{m+1})$.
Moreover, assuming that the vertex normals in \eqref{eq:nuhomegah} satisfy
\begin{equation}
\vec\nu_p^m(\vec q)\neq 0\quad\forall \vec q\in Q_\Gamma^m,\label{eq:vnormalassump}
\end{equation}
then the linear system \eqref{eqn:Xfull} admits a unique solution $\vec X^{m+1}\in [V^h(\Gm)]^d$ and $\lambda^{m+1}_\Gamma=0 \in V^h(\Gm)$.

\end{thm}
\begin{proof}
Since \eqref{eqn:fdnv} and \eqref{eqn:fdsf} form a linear system where the number of unknowns equals the number of equations, it suffices to show that the corresponding homogeneous system has only the zero solution. We thus consider the following homogeneous system by finding $(\vec U, P, P_\Gamma, F_\Gamma, \varkappa, \mathscr{V})\in \bU^m(\vec 0)\times\bP^m\times [V^h(\Gm)]^4$ and $P_{\rm sing}\in\bR$ such that
\begin{subequations}
\begin{align}
\label{eq:homonv1}
&\frac{1}{2\,\ttau}\left[\ipD{\rho^m\,\vec U, ~\vec\xi^h} + \ipD{\rho^{m-1}(\Lambda_{m}^{m-1}\vec U), ~\Lambda_m^{m-1}\vec\xi^h}\right]-\ipD{P,~\nabla\cdot\vec\xi^h}+ 2\ipD{\mu^m\,\mat{D}(\vec U),~\mat{D}(\vec\xi^h)}\nn\\
&\qquad\qquad -P_{\rm sing}\ipd{\vec\nu_p^m,~\vec\xi^h}_{\Gm}^h +\mathscr{A}\left(\rho^m[\Lambda_{m-1}^m\vec U^m-\vec W^{m}];\vec U,\vec\xi^h\right)\nn\\
&\qquad\qquad +\frac{\rho_\Gamma}{2\,\ttau}\left[\ipd{\vec U,~\vec\xi^h}_{\Gm}  + \ipd{\Pi_m^{m-1}\vec U,~\Pi_m^{m-1}\vec\xi^h}_{\Gamma^{m-1}}\right]- \ipd{P_\Gamma,~\nabs\cdot[\vec\pi_1^m(\vec\xi^h)]}_{\Gm}\nn\\
&\qquad\qquad   +2\mu_\Gamma\ipd{\mat{D}_s^m(\vec U),~\mat{D}_s^m(\vec\xi^h)}_{\Gm} = \alpha \ipd{F_\Gamma\vec\nu^m,~\vec\xi^h}_{\Gm}^h; \\[0.5em]
\label{eq:homonv2}
&\ipD{\nabla\cdot\vec U,~q^h} = 0;\qquad\mbox{with}\quad \ipd{\vec\nu_p^m,~\vec U}_{\Gm}^h=0;\\[0.5em]
\label{eq:homonv3}
&\ipd{\nabs\cdot\vec U,~q^h_{\Gamma}}_{\Gamma^m}=0;\\[0.5em]
\label{eq:homosf1}
&\ipd{F_\Gamma,~\varphi^h}_{{\Gm}}^h-\ipd{\nabs\varkappa,~\nabs\varphi^h}_{\Gm} + \ipd{\mathcal{W}^m\,\varkappa,~\varphi^h}_{\Gm}=\frac{1}{2}\ipd{(\Pi_{m-1}^m\varkappa^m)^2\,\varkappa,~\varphi^h}_{\Gm};\\[0.5em]
\label{eq:homosf2}
&\frac{1}{2\,\ttau}\Bigl[\ipd{\varkappa, ~\chi^h}_{\Gamma^m} + \ipd{\Pi_{m}^{m-1}\varkappa,~\Pi_{m}^{m-1}\chi^h}_{\Gamma^{m-1}}\Bigr]-\mathscr{A}_{\Gamma^m}(\mathscr{\vv V}^{m}; \varkappa,\chi^h)+\ipd{\nabs\mathscr{V}, \nabs\chi^h}_{\Gamma^m}\nn\\
&\qquad\quad-\ipd{\mathscr{ V}\,\mathcal{W}^m,~\chi^h}_{\Gamma^m}  + \frac{1}{2}\ipd{\mathscr{V},~(\Pi_{m-1}^{m}\varkappa^m)^2\,\chi^h}_{\Gamma^m} =0;\\[0.5em]
&\ipd{\mathscr{V},~\zeta^h}_{\Gm}^h-\ipd{\vec U\cdot\vec\nu^m,~\zeta^h}_{\Gm}^h=0;
\label{eq:homosf3}
\end{align}
\end{subequations}
for $(\vec\xi^h, q^h, q_\Gamma^h, \varphi^h, \chi^h, \zeta^h)\in\bU^m(\vec 0)\times \bP^m\times[V^h(\Gm)]^4$.

We then choose $\vec\xi^h = \vec U$ in \eqref{eq:homonv1}, $q^h = P$ in \eqref{eq:homonv2}, $q_\Gamma^h = P_\Gamma$ in \eqref{eq:homonv3}, $\varphi^h = \mathscr{V}$ in \eqref{eq:homosf1}, $\chi^h = \varkappa$ in \eqref{eq:homosf2} and $\zeta^h = F_\Gamma$ in \eqref{eq:homosf3}. Combining these equations, we obtain
\begin{align}
&\frac{1}{2\,\ttau}\left[\ipD{\rho^m\,\vec U, ~\vec U} + \ipD{\rho^{m-1}(\Lambda_{m}^{m-1}\vec U), ~\Lambda_m^{m-1}\vec U}\right]+2\ipD{\mu^m\,\mat{D}(\vec U),~\mat{D}(\vec U)}\nonumber\\
&+  \frac{\rho_\Gamma}{2\,\ttau}\left[\ipd{\vec U, \vec U}_{\Gm}^h + \ipd{\Pi_m^{m-1}\vec U,\Pi_m^{m-1}\vec U}_{\Gamma^{m-1}}^h\right] + 2\mu_\Gamma\ipd{\mat{D}_s^m(\vec U), \mat{D}_s^m(\vec U)}_{\Gm}\nn\\
&+ \frac{\alpha}{2\,\ttau}\left[\ipd{\varkappa,~\varkappa}_{\Gm} + \ipd{\Pi_m^{m-1}\varkappa, \Pi_m^{m-1}\varkappa}_{\Gamma^{m-1}}\right]=0.
\end{align}
Then, by $\mathscr{H}^{d-1}(\partial_1\Omega)>0$ with Korn's inequality and $\alpha>0$, we obtain that $\vec U=\vec 0$ and $\varkappa=0$. Inserting    $\vec U=\vec 0$ and $\varkappa=0$ into \eqref{eq:homosf1} and \eqref{eq:homosf3} leads to $F_\Gamma = 0$ and $\mathscr{V}=0$.

Finally, using $\vec U=\vec 0$ and $F_\Gamma=0$ in \eqref{eq:homonv1} and on recalling the inf-sup condition \eqref{eq:infsup}, we obtain that $P=0\in\bP^m$, $P_\Gamma=0\in V^h(\Gm)$ and $P_{\rm sing}=0$. Therefore, we have proven that the linear systems \eqref{eqn:fdnv} and \eqref{eqn:fdsf} have a unique solution.

We next consider the solvability of the linear system \eqref{eqn:Xfull}. Similarly, we consider its homogeneous variant: find $\vec X\in [V^h(\Gm)]^d$ and $\lambda_\Gamma\in V^h(\Gm)$ such that
\begin{align}
 \label{eq:homoXfull1}
&\ipd{\vec X\cdot\vec\nu_p^m,~\eta^h}_{\Gm}^h = 0\qquad\forall\eta^h\in V^h(\Gm),\\
\label{eq:homoXfull2}
&\ipd{\mat{P}_{\Gm}\vec X,~\vec\eta^h}_{\Gm}^h + \ipd{\lambda_\Gamma\,\vec\nu_p^m,~\vec\eta^h}_{\Gm}^h = 0\qquad\forall\vec\eta^h\in [V^h(\Gm)]^d.
\end{align}
Choosing $\eta^h = \lambda_\Gamma$ in \eqref{eq:homoXfull1} and $\vec\eta^h =\vec X$ in \eqref{eq:homoXfull2}, we recall the mass-lumped inner product \eqref{eq:tprule} to obtain
\begin{equation}
0=\ipd{\mat{P}_{\Gm}\vec X,~\vec X}_{\Gm}^h = \ipd{\mat{P}_{\Gm}\vec X,~\mat{P}_{\Gm}\vec X}_{\Gm}^h.\label{eq:homoX1}
\end{equation}
We assume $\pi^m$ is the standard interpolation operator $\pi^m: [C(\Gm)]\to V^h(\Gm)$ and then set $\eta^h = \pi^m(\vec X\cdot\vec\nu_p^m)$ in \eqref{eq:homoXfull1} to obtain
\begin{align}
0=\ipd{\vec X\cdot\vec\nu_p^m, ~\vec X\cdot\vec\nu_p^m}_{\Gm}^h = \ipd{\vec X\cdot\vec\nu_p^m,~\pi^m(\vec X\cdot\vec \nu_p^m)}_{\Gm}^h.\label{eq:homoX2}
\end{align}
Combining \eqref{eq:homoX1} and \eqref{eq:homoX2} implies that $\ipd{\vec X, \vec X}_{\Gm}^h = 0$, and thus $\vec X =\vec 0$ as desired.

Inserting $\vec X=\vec 0$ in \eqref{eq:homoXfull2} and setting $\vec\eta^h = \vec\nu_p^m$ yields that
\begin{equation}
\ipd{\lambda_\Gamma\,\vec\nu_p^m,~\vec\nu_p^m}
_{\Gm}^h = 0,
 \end{equation}
which implies that $\lambda_\Gamma=0\in V^h(\Gm)$ on recalling the assumption \eqref{eq:vnormalassump}. Thus, \eqref{eqn:Xfull} has a unique solution
$(\vec X^{m+1}, \lambda^{m+1}_\Gamma)$.
The fact that $\lambda^{m+1}_\Gamma=0\in V^h(\Gm)$ follows from choosing
$\vec\eta = \vec\nu_p$ in \eqref{eq:Xfull2} and the assumption
\eqref{eq:vnormalassump}. 
\end{proof}

We have the following theorem, which shows that the discrete solution satisfies an unconditional energy stability estimate.
\begin{thm}[stability estimate] Let $(\vec U^{m+1}, P^{m+1}, P_\Gamma^{m+1}, P^{m+1}_{\rm sing}, F_\Gamma^{m+1}, \varkappa^{m+1}, \mathscr{V}^{m+1})$ be a solution to the linear system \eqref{eqn:fdnv} and \eqref{eqn:fdsf} for each $m\geq 0$.
In the case $\vec g = \vec 0$, the discrete solution obeys the following energy stability estimate:
\begin{align}\label{eq:dES}
&\mathcal{E}(\rho^m, \vec U^{m+1}, \Gamma^m, \varkappa^{m+1})+ 2\ttau\norm{\sqrt{\mu^m}\mat{D}(\vec U^{m+1})}_0^2+ 2\ttau\,\mu_\Gamma\norm{\mat{D}_{s}^m(\vec U^{m+1})}_{\Gm,0}^2\nn\\
&\quad+ \frac{\rho_+\,\ttau}{2}\ipd{\vec U^m\cdot\vec n_{_{\partial\Omega}},~\vec U^m\cdot\vec U^{m+1}}_{\partial_2\Omega}\leq\mathcal{E}(\rho^{m-1}, \vec U^{m}, \Gamma^{m-1}, \varkappa^{m}),
\end{align}
where we define
\begin{equation*}
\mathcal{E}(\rho, \vec u, \Gamma, \varkappa):= \frac{1}{2}(\rho\vec u,~\vec u) + \frac{1}{2}\ipd{\rho_\Gamma\vec u,~\vec u}_{\Gamma} + \frac{\alpha}{2}\ipd{\varkappa-\bkap,~\varkappa-\bkap}_{\Gamma}.
\end{equation*}
\end{thm}

\begin{proof}
We set $\xi^h = \ttau\vec U^{m+1}$ in \eqref{eq:fdnv1}, $q = P^{m+1}$ in \eqref{eq:fdnv2}, $q_\Gamma = P_\Gamma^{m+1}$ in \eqref{eq:fdnv3}, $\varphi^h = \mathscr{V}^{m+1}$ in \eqref{eq:fdsf1}, $\chi^h = \ttau(\varkappa^{m+1}-\bkap)$ in \eqref{eq:fdsf2} and $\zeta^h = F_\Gamma^{m+1}$ in \eqref{eq:fdsf3}. Combining these equations then leads to
\begin{align}
&\frac{1}{2}\left[\ipD{\rho^m \vec U^{m+1}, ~\vec U^{m+1}} + \ipD{\rho^{m-1}\,\Lambda_m^{m-1}\vec U^{m+1},~\Lambda_m^{m-1}\vec U^{m+1}}\right] \nn\\
& + \frac{\rho_\Gamma}{2}\left[\ipd{\vec U^{m+1}, \vec U^{m+1}}_{\Gm}^h + \ipd{\Pi_m^{m-1}\vec U^{m+1},\Pi_m^{m-1}\vec U^{m+1}}_{\Gamma^{m-1}}^h\right] \nn\\
& + 2\ttau\ipD{\mu^m\mat{D}(\vec U^{m+1}),~\mat{D}(\vec U^{m+1})}+ 2\mu_\Gamma\ttau\ipd{\mat{D}_s^m(\vec U^{m+1}), \mat{D}_s^m(\vec U^{m+1})}_{\Gm}\nn\\
&+ \frac{\alpha}{2}\left[\ipd{\varkappa^{m+1}-\bkap,~\varkappa^{m+1}-\bkap}_{\Gm} + \ipd{\Pi_m^{m-1}\varkappa^{m+1}-\bkap, \Pi_m^{m-1}\varkappa^{m+1}-\bkap}_{\Gamma^{m-1}}\right] \nn\\
&=\ipD{\rho^{m-1}\vec U^m,~\Lambda_m^{m-1}\vec U^{m+1}} + \rho_\Gamma\ipd{\vec U^m,~\Pi_m^{m-1}\vec U^{m+1}}_{\Gamma^{m-1}}\nn\\
&\qquad   - \frac{\ttau}{2}\ipd{\rho_+\vec U^m\cdot\vec n,~\vec U^m\cdot\vec U^{m+1}}_{\partial_2\Omega} + \alpha\ipd{\varkappa^m-\bkap,~\Pi_m^{m-1}\varkappa^{m+1}-\bkap}_{\Gamma^{m-1}}.\label{eq:stab1}
\end{align}

We can then rewrite \eqref{eq:stab1} into
\begin{align}
&\frac{1}{2}\left[\ipD{\rho^m\vec U^{m+1}, ~\vec U^{m+1}}+ \rho_\Gamma\ipd{\vec U^{m+1},\vec U^{m+1}}_{\Gm}^h + \alpha\ipd{\varkappa^{m+1}-\bkap,~\varkappa^{m+1}-\bkap}_{\Gm}\right]\nn\\
&+2\ttau\norm{\sqrt{\mu^m}\mat{D}(\vec U^{m+1})}_0^2 + 2\ttau\,\mu_\Gamma\norm{\mat{D}_s^m(\vec U^{m+1})}_{\Gamma^m,0}^2+\frac{\rho_+\ttau}{2}\ipd{\vec U^m\cdot\vec n,~\vec U^m\cdot\vec U^{m+1}}_{\partial_2\Omega}\nn\\
&\leq \ipD{\rho^{m-1}\vec U^m,~\Lambda_m^{m-1}\vec U^{m+1}} - \frac{1}{2}\ipD{\rho^{m-1}\,\Lambda_m^{m-1}\vec U^{m+1},~\Lambda_m^{m-1}\vec U^{m+1}}\nn\\
&\quad  + \rho_\Gamma\left[\ipd{\vec U^m,~\Pi_m^{m-1}\vec U^{m+1}}_{\Gamma^{m-1}}^h - \frac{1}{2}\ipd{\Pi_m^{m-1}\vec U^{m+1},\Pi_m^{m-1}\vec U^{m+1}}_{\Gamma^{m-1}}^h\right]\nn\\
&\quad+ \alpha\left[\ipd{\varkappa^m-\bkap,~\Pi_m^{m-1}\varkappa^{m+1}-\bkap}_{\Gamma^{m-1}} - \frac{1}{2}\ipd{\Pi_m^{m-1}\varkappa^{m+1}-\bkap,~\Pi_m^{m-1}\varkappa^{m+1}-\bkap}_{\Gamma^{m-1}}\right]. \label{eqn:stab2}
\end{align}
We now apply the inequality $\vec a\cdot\vec b- \frac{1}{2} |\vec b|^2 \leq \frac{1}{2}|\vec a|^2$ to the right-hand side of \eqref{eqn:stab2} at each line, e.g.,
\begin{align}
&\ipD{\rho^{m-1}\vec U^m,~\Lambda_m^{m-1}\vec U^{m+1}} - \frac{1}{2}\ipD{\rho^{m-1}\,\Lambda_m^{m-1}\vec U^{m+1},~\Lambda_m^{m-1}\vec U^{m+1}}\leq \frac{1}{2}\ipD{\rho^{m-1}\vec U^m, \vec U^m}.\nn
\end{align}
This gives rise to \eqref{eq:dES} straightforwardly.
\end{proof}

\begin{rem}
It is worth noting that $\vec X^{m+1}$ and $\Gamma^{m+1}$ do not enter the stability estimate \eqref{eq:dES}, and so the estimate is independent of the update step \eqref{eqn:Xfull}. Of course, the interface must be advanced according to \eqref{eqn:Xfull} to recover a consistent approximation of the geometric evolution. Otherwise one may obtain an evolution that is energetically stable but geometrically incorrect.
\end{rem}

\section{Numerical results from the ALE approach}\label{sec:num}

\subsection{Solution strategies}

The linear system \eqref{eqn:fdnv} and \eqref{eqn:fdsf} from our ALE approximations can be written in matrix form as
\begin{equation}
\begin{pmatrix}
 \mat{\vec B}_\Omega^m & \mat{\vec{\mathcal{C}}}^m & 0 & 0 & -\alpha\,\Nbulk^m \\
 (\mat{\vec {\mathcal{C}}}^m)^T & 0 & 0 & 0 & 0 \\
(\Nbulk^m)^T & 0 & 0 & -\mat{M}_\Gamma^m & 0 \\
0 & 0 & \mat{B}_\Gamma^m & \mat{A}^m_\Gamma  & 0\\
0 & 0 &  - \mat A^m_\Gamma & 0 & \mat M_\Gamma^m  
\end{pmatrix} 
\begin{pmatrix} \vec U^{m+1} \\  \tilde{P}^{m+1} \\ \varkappa^{m+1} \\ 
\mathscr{V}^{m+1} \\ F^{m+1}_\Gamma \end{pmatrix}
=
\begin{pmatrix} \vec b^m \\ 0 \\ 0 \\ b_\Gamma^m
\\
c_\Gamma^m 
\end{pmatrix} \,,
\label{eq:lin2}
\end{equation}
where $\tilde{P}^{m+1}=(P^{m+1}, P_\Gamma^{m+1}, P_{\rm sing}^{m+1})$, and with a slight abuse of notation, we denote by $(\vec U^{m+1}, P^{m+1}, \varkappa^{m+1},\mathscr{V}^{m+1}, F_\Gamma^{m+1})$ the coefficients of these finite element functions with respect to the standard bases of the corresponding finite element spaces.

For the solution of (\ref{eq:lin2}), a Schur complement approach similar to
\cite{BGN15stable, BGN16sfluidic} can be used.
In particular, the Schur approach for eliminating
$(\mathscr{V}^{m+1}, F^{m+1}_\Gamma, \varkappa^{m+1})$ from
(\ref{eq:lin2}) can be obtained as follows. Let
\begin{equation*} 
\mat{\Theta}_\Gamma^m:= \begin{pmatrix}
0 & -\mat M_\Gamma^m & 0 \\
 \mat{B}^m_\Gamma & \mat{A}_\Gamma^m  & 0\\
 - \mat{A}_\Gamma^m & 0 & \mat M^m_\Gamma  
\end{pmatrix} \,.
\end{equation*}
Then, (\ref{eq:lin2}) can be reduced to
\begin{subequations}
\begin{align} \label{eq:schur2}
&
\begin{pmatrix}
\mat {\vec B}^m_\Omega + \alpha\,\mat {\vec T}^m_\Omega
& \mat{\vec {\mathcal{C}}}^m \\
(\mat{\vec {\mathcal{C}}}^m)^T & 0 
\end{pmatrix}
\begin{pmatrix}
\vec U^{m+1} \\ \tilde{P}^{m+1} 
\end{pmatrix}
= \begin{pmatrix}
\vec b^m + \alpha\,\vec c^m \\
0
\end{pmatrix}
\end{align}
and
\begin{equation} \label{eq:schurb2}
\begin{pmatrix}
\varkappa^{m+1} \\ \mathscr{V}^{m+1} \\ F^{m+1}_\Gamma 
\end{pmatrix}
 = (\Theta^m_\Gamma)^{-1}\,
\begin{pmatrix}
-(\Nbulk^m)^T\,\vec U^{m+1} \\ 
b_\Gamma^m \\
c_\Gamma^m 
\end{pmatrix}
\,,
\end{equation}
\end{subequations}
where we have used the definitions
\begin{equation*}
\mat{\vec T}_\Omega^m = (0\ 0\ \Nbulk^m)\,(\Theta^m_\Gamma)^{-1}\,
{ \begin{pmatrix} (\Nbulk^m)^T \\ 0 \\ 0 \end{pmatrix}},\qquad \vec c^m  = (0\ 0\ \Nbulk^m)\,(\Theta^m_\Gamma)^{-1}\,
 \begin{pmatrix}
0 \\ 
b_\Gamma^m \\
c_\Gamma^m
\end{pmatrix}.
\end{equation*}
The system \eqref{eq:schur2} can be solved with a preconditioned Krylov subspace solver, while the system \eqref{eq:schurb2} can be solved with the help of a sparse LU factorization. In the second stage, the resulting linear system for \eqref{eqn:Xfull} can also be solved with a sparse LU factorization.

Unless stated otherwise, we always choose $\bar{\varkappa}=0$, $\partial_1\Omega=\partial\Omega$ with $\vec g=\vec 0$, $\vec u_0=\vec 0$. In addition, we employ the $\mbox{P2-(P1+P0)}$ element pair, as defined in \eqref{eq:P2P1P0}. For the initial discrete curvature $\varkappa^0$, we solve for it using the BGN method with a zero normal velocity and fixed boundary. That is, we find $(\delta\vec  Y^0, \kappa^0)\in [V^h(\Gamma_Y)]^d\times V^h(\Gamma_Y)$ such that
\begin{subequations}\label{eqn:ic}
	\begin{align}
		&\ipd{\delta \vec Y^0\cdot\vec\nu_Y, \xi^h}_{\Gamma_Y}^h =0\qquad\forall\xi^h\in V^h(\Gamma_Y),\\
		&\ipd{\kappa^0\,\vec\nu_Y,~\vec\eta^h}_{\Gamma_Y}^h + \ipd{\nabs(\vec\id+\delta\vec  Y^0),~\nabs\vec\eta^h}_{\Gamma_Y}=0\qquad\forall\vec\eta^h\in [V^h(\Gamma_Y)]^d,
	\end{align}
	\end{subequations}
where $\vec\nu_Y$ is the unit normal of $\Gamma_Y$, which follows \eqref{eq:vG} similarly. We next set $\vec X^0 = \vec\id|_{\Gamma_Y} + \delta\vec Y^0$ with $\Gamma^0 = \vec X^0(\Gamma_Y)$ and $\varkappa^0=\kappa^0$ as the required initial data.

\subsection{Numerical results}

\vspace{0.2cm}
\noindent
{\bf Example 1}: We start with a convergence test for our scheme in the first example. We consider $\Omega=(0,1)\times[0, 1.6]$ and choose
\begin{equation}
\rho_+=10,\quad\rho_-=1,\quad\mu_+=1, \quad\mu_-=0.1,\quad \rho_\Gamma=1,\quad\mu_\Gamma=1,\quad\alpha=1.\label{eq:station2}
\end{equation}
Initially, the fluid interface is given by an ellipse satisfying $\frac{(x-0.5)^2}{0.1^2}+\frac{(y-0.8)^2}{0.4^2}=1$. For two closed curves $\Gamma_1$ and $\Gamma_2$, we measure their difference by the area of the symmetric difference between the region $\Omega_1$ and $\Omega_2$ \cite{Zhao2021energy}
\begin{equation*}
{\rm MD}(\Gamma_1,\Gamma_2) = |(\Omega_1\setminus\Omega_2)\cup(\Omega_2\setminus\Omega_1)| = |\Omega_1| + |\Omega_2|-2 |\Omega_1\cap\Omega_2|,
\end{equation*}
where $\Omega_i$ is the region enclosed by $\Gamma_i$, and $|\Omega|$ stands for the area of the region $\Omega$. We then define the numerical errors
\begin{equation*}
e_{h,\ttau}(t) = {\rm MD }(\Gamma_{h,\ttau}(t), \Gamma_{2\,h, 4\ttau}(t)),
\end{equation*}
where $\Gamma^h_{h,\ttau}(t) = \vec X_{h,\ttau}(\Gamma^m, t)$ for $t\in [t_m, t_{m+1}]$ and $\vec X_{h,\ttau}(\cdot, t)$ is defined via
\begin{equation}
		\vec X_{h,\ttau} (\vec q, t) = \frac{t_{m+1}-t}{\ttau}\vec q + \frac{t-t_m}{\ttau}\vec X^{m+1}(\vec q),\quad\forall\vec q\in Q_\Gamma^m, \quad t\in[t_m, t_{m+1}].\label{eq:linertime}
	\end{equation}

We also introduce the quantities to measure the loss of area and volume as:
\begin{align}
\Delta A_{\infty} &= \max_{m=1,\cdots, M} \bigl||\Gamma^m| - |\Gamma(0)|\bigr|,\nn\\
\Delta V_{\infty} & = \max_{m=1,\cdots, M}|\vol(\Gamma^m) - \vol(\Gamma(0))|.\nn
\end{align}
The numerical errors are reported in Table~\ref{tab:stat2}. Here we observe that as the mesh size and time step are refined, the numerical errors and the loss of surface area $\Delta A_{\infty}$ and enclosed volume $\Delta V_\infty$ become smaller, which verifies the numerical convergence of our proposed method. 

\begin{table}[!htp]
\centering
\def\temptablewidth{0.75\textwidth}
\vspace{-2pt}
\caption{Convergence experiments for a time-dependent solution over the interval $[0,1]$ with the parameters as in \eqref{eq:station2}, where $h_0=\frac{1}{16}$, $\ttau_0=0.01$ with $h = \frac{1}{J_\Gamma}$. }\label{tab:stat2}
{\rule{\temptablewidth}{1pt}}
\begin{tabular}{c|cccc}
$(h,\ttau)$   & $e_{h,\ttau}(t=0.05)$    &  $e_{h,\ttau}(t=0.20)$  &$\Delta A_\infty$   &$\Delta V_\infty$\\ \hline
$(h_0,\ttau_0)$  &-   & - &1.04E-2    &3.40E-3    \\ \hline
$(\frac{h_0}{2},\frac{\ttau_0}{4})$ &1.06E-2    &1.05E-2  &1.31E-2 &1.00E-3 \\ \hline
$(\frac{h_0}{2^2},\frac{\ttau_0}{4^2})$ &2.50E-3  &2.90E-3  &7.80E-3 &2.91E-4 \\ \hline
$(\frac{h_0}{2^3},\frac{\ttau_0}{4^3})$&1.20E-3  &1.11E-3  &2.80E-3 &7.54E-5 \\ \hline
 
\end{tabular}
{\rule{\temptablewidth}{1pt}}
\end{table}

\vspace{0.2cm}
\noindent{\bf Example 2}: We next start with an initial shape in the form of a smooth letter ``C''. We choose the computational domain $\Omega=(0,2)^2$ and fix the parameters
\begin{equation*}
\rho_\pm=1,\quad\mu_\pm=1,\quad \mu_\Gamma=1,\quad\alpha=1,
\end{equation*}
and compare two simulations with $\rho_\Gamma=0$ and $\rho_\Gamma=1$. For the initial computational meshes, we use $J_\Gamma=128$, $J_\Omega=2860, K_\Omega=1451$ and set $\ttau=5\times 10^{-4}$. We also introduce the discrete quantities
\begin{align*}
\Delta V^m = \frac{\vol(\Gamma^m)-\vol(\Gamma^0)}{\vol(\Gamma^0)},\qquad \Delta A^m = \frac{|\Gamma^m|-|\Gamma^0|}{|\Gamma^0|},\qquad m\geq 0.
\end{align*}
The numerical results are presented in Figs.~\ref{fig:cshape0} and~\ref{fig:cshape1}. In the latter case, the two arms of the vesicle oscillate up and down due to inertia effects. This behaviour is absent in the case $\rho_\Gamma=0$, where the evolution is purely dissipative and the vesicle relaxes monotonically towards its equilibrium configuration without noticeable oscillations. In both cases, we observe good preservation of volume and surface area, along with a decay of the discrete energy.

\begin{figure}[!htp]
\centering
\includegraphics[width=0.9\textwidth]{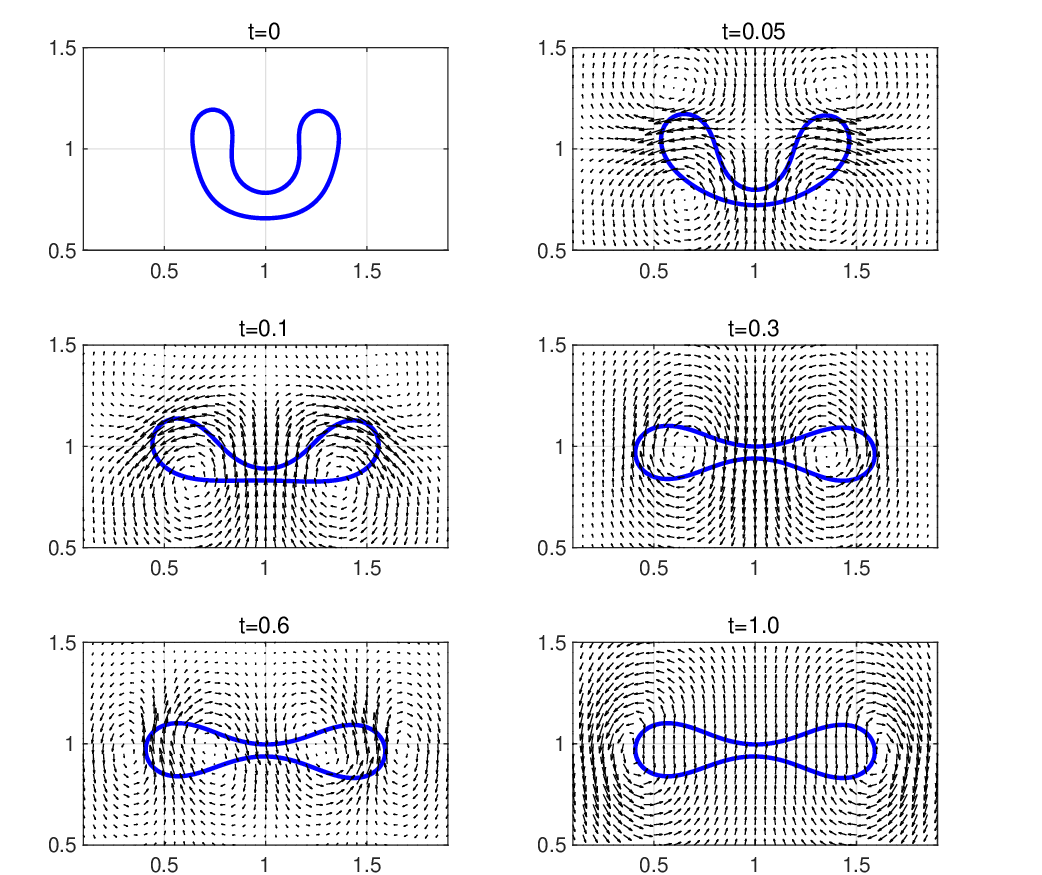}
\includegraphics[width=0.9\textwidth]{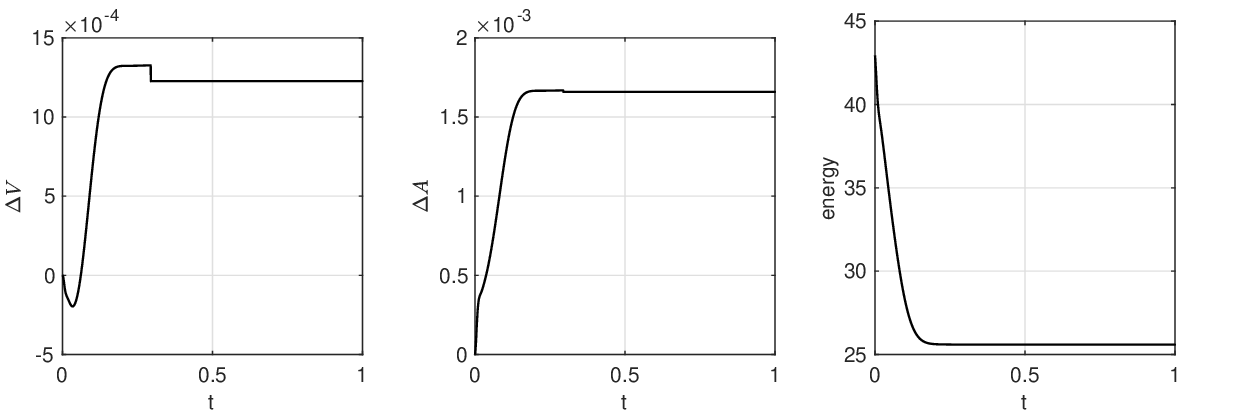}
\caption{Evolution for a smooth letter ``C'' in the case of $\rho_\Gamma=0$, where the plots of $\Gamma^m$ together with the velocity fields are shown at $t=0, 0.05, 0.3, 0.6, 1.0$. Below are the plots of the discrete quantities. }
\label{fig:cshape0}
\end{figure}

\begin{figure}[!htp]
\centering
\includegraphics[width=0.9\textwidth]{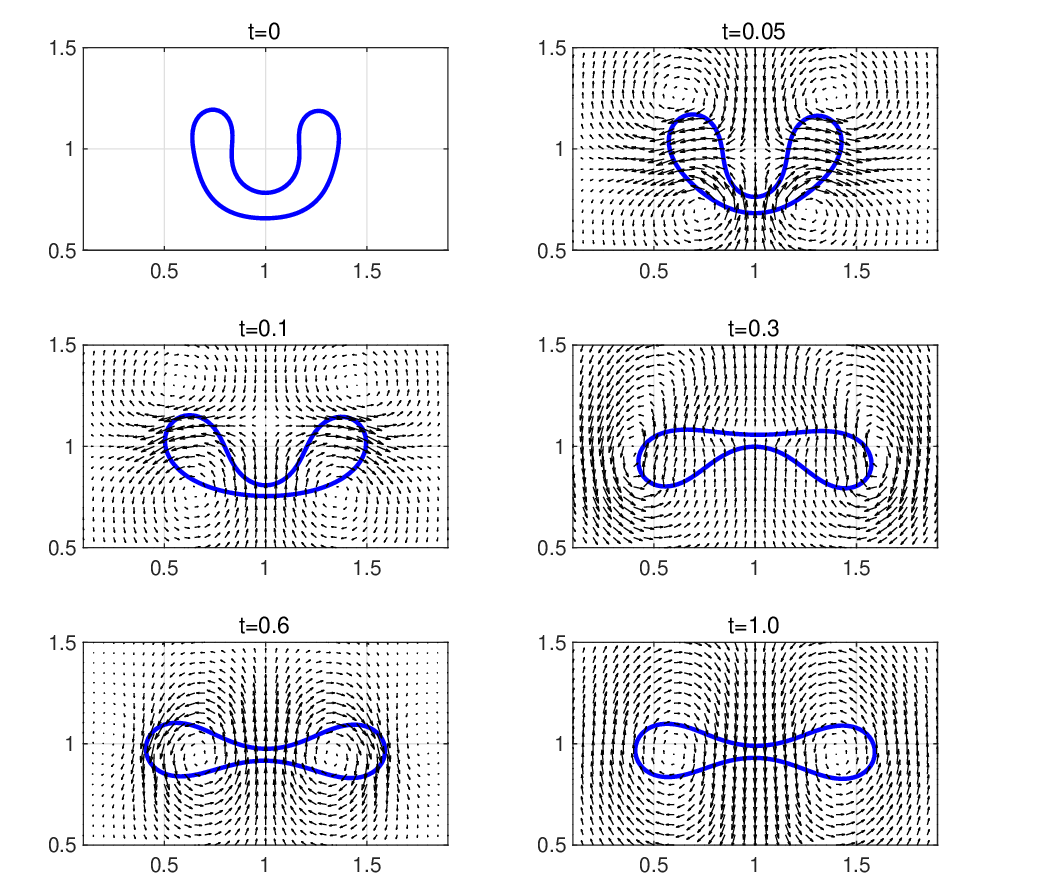}
\includegraphics[width=0.9\textwidth]{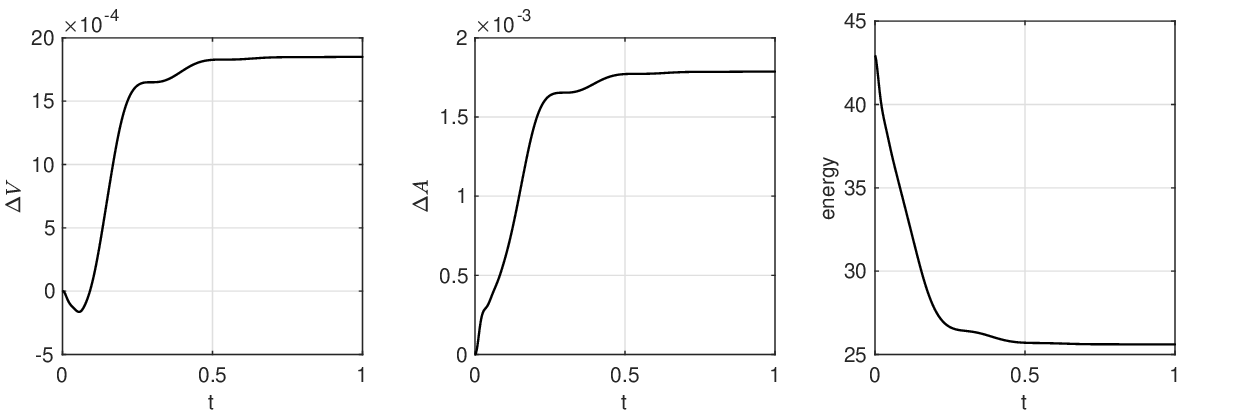}
\caption{Evolution for a smooth letter ``C'' in the case of $\rho_\Gamma=1$, where the plots of $\Gamma^m$ together with the velocity fields are shown at $t=0, 0.05, 0.3, 0.6, 1.0$. Below are the plots of the discrete quantities. }
\label{fig:cshape1}
\end{figure}

\begin{figure}[!htp]
\centering
\includegraphics[width=0.9\textwidth]{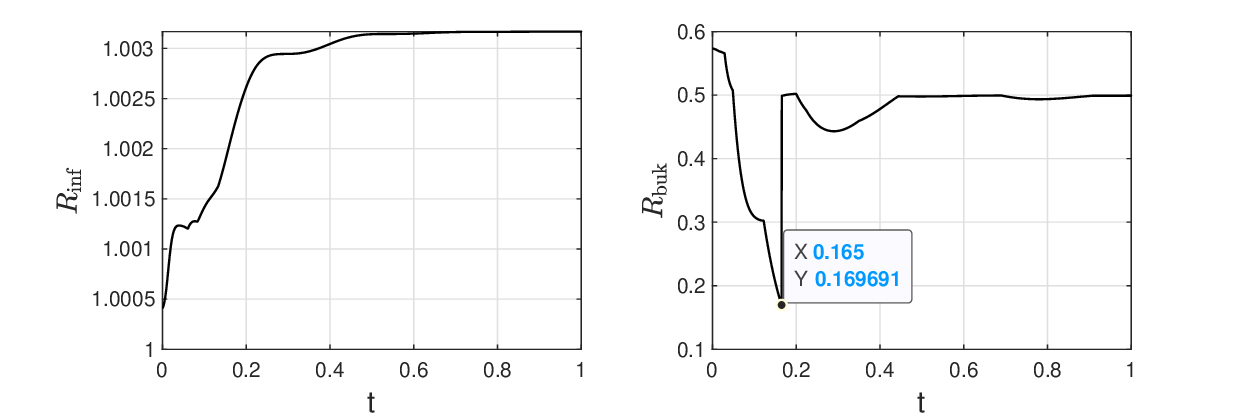}
\caption{The time history of the discrete quantities $R_{\rm inf}$ and $R_{\rm buk}$ in the evolution for a smooth letter ``C'' with $\rho_\Gamma=1$. Here, remeshing of the bulk mesh is employed at time $t=0.165$. }\label{fig:meshQ}
\end{figure}

\begin{figure}[!htp]
\centering
\includegraphics[width=0.9\textwidth]{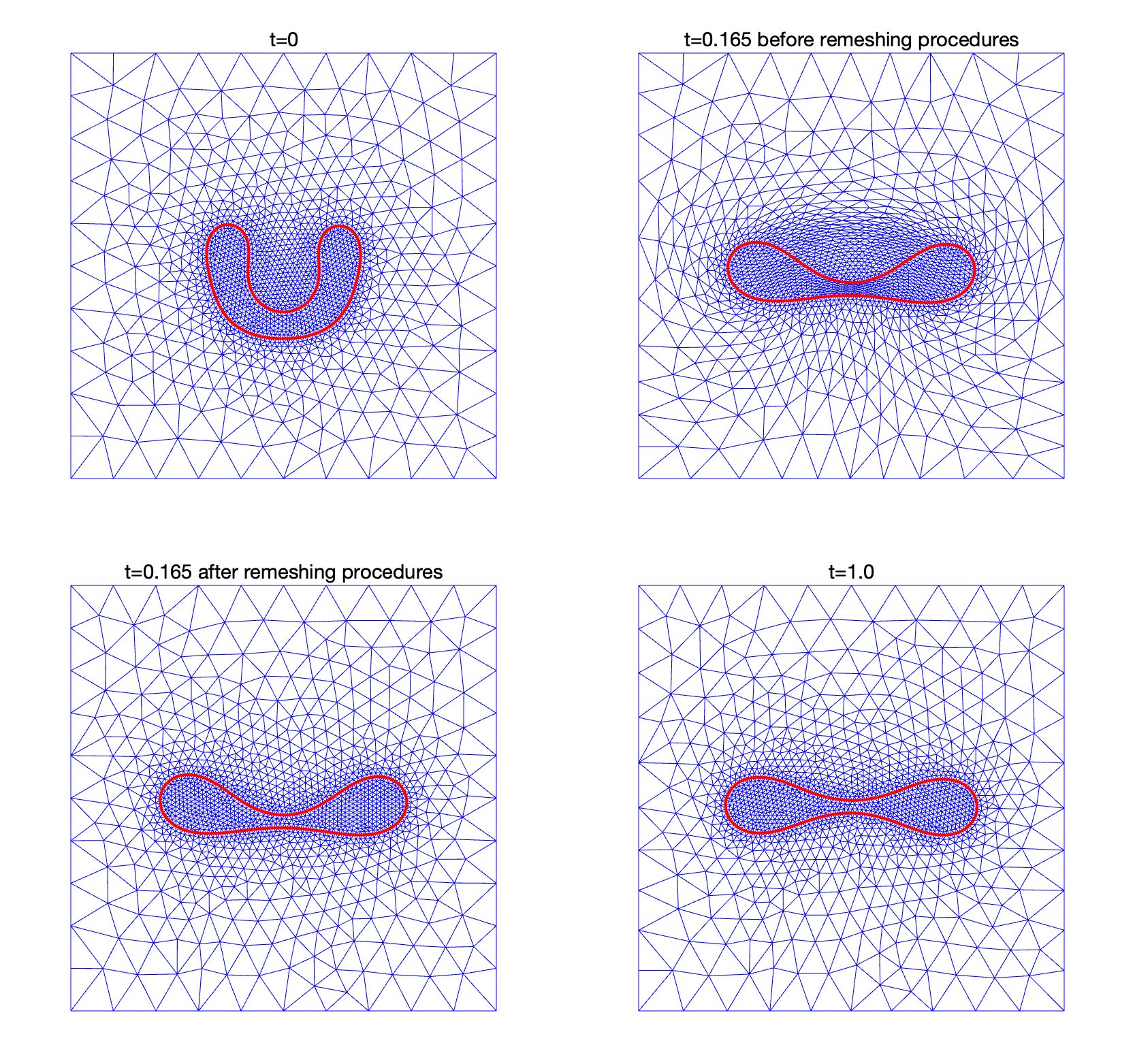}
\caption{Several snapshots of the computational meshes. Upper panel: the old meshes with $J_\Gamma=128, J_\Omega=2860, K_\Omega=1451$. Lower panel: the new meshes with $J_\Gamma=128, J_\Omega=2716, K_\Omega=1379$.}
\label{fig:remeshing}
\end{figure}

To monitor the quality of the computational mesh, we also introduce the discrete quantities
\begin{equation*}
R_{\rm buk}|_{t_m}=\min_{\sigma\in\mathscr{T}^m}\min_{\theta\in \measuredangle{(\sigma)}},\qquad  R_{\rm inf}|_{t_m} := \frac{\max_{j=1}^{J_\Gamma}|\sigma_j^m|}{\min_{j=1}^{J_\Gamma}|\sigma_j^m|},
\end{equation*}
where $\measuredangle{(\sigma)}$ is the set of all the angles of the simplex $\sigma$. Here, $R_{\rm buk}$ and $R_{\rm inf}$ measure the quality of the bulk mesh and interface mesh, respectively.

The time histories of the discrete quantities are shown in Fig.~\ref{fig:meshQ}. We observe that the interface mesh is well preserved, as $R_{\rm inf}$ remains close to 1. This is because the interface evolves according to the local fluid velocity, which is surface incompressible. Moreover, our moving bulk mesh approach generally works smoothly, except when the interface undergoes strong deformations. Here, we observe that $R_{\rm buk}$ gradually decreases at the early stage, indicating a deterioration in the quality of the bulk mesh. To prevent this, we monitor $R_{\rm buk}$ every ten time steps, and if it falls below $\frac{\pi}{18}$, we regenerate the bulk mesh based on the interface mesh $\Gamma^m$. After remeshing, the fluid velocity and bulk mesh velocity are appropriately interpolated onto the new mesh.

As shown in Fig.~\ref{fig:meshQ}, after remeshing, the quality indicator $R_{\rm buk}$ of the bulk mesh is improved immediately and then remains at approximately 0.5 at later times. This indicates that only one remeshing step is required for the bulk mesh in our current example, which demonstrates the efficiency of the approach. We also visualize the computational meshes at several times in Fig.~\ref{fig:remeshing}.

\vspace{0.2cm}
\noindent{\bf Example 3}: In this example, we consider the dynamics of an initial interface in the form of an ellipse with axis lengths of 1 and 2.5 in a shearing flow. In particular, we set the initial interface as $\frac{(x-2)^2}{0.5^2} + \frac{(y-2)^2}{1.25^2}=1$ and consider $\Omega=(0, 4)^2$ with $\partial_1\Omega=(0,4)\times\{0,4\}$ and $\vec g(\vec z) = (z_2, 0)^T$ for $\vec z=(z_1, z_2)^T$. We also use the initial data $\vec u_0(\vec z)=\eta(z_2)\vec e_1$, where $\eta$ is a continuous piecewise linear function with $\eta(0)= -2$, $\eta(4)=0$ with $\eta(s) = 0$ for $0.5<s<3.5$; see \cite{BGN16sfluidic}. We use the parameters
\begin{equation}
\rho_\pm = 1, \quad \rho_\Gamma=1,\quad \mu_+=1,\quad \mu_\Gamma=1, \quad \alpha = 0.05,
\end{equation}
and compare two simulations with $\mu_-=1$ and $\mu_-=10$. For the initial computational mesh, we fix $J_\Gamma=128$, $J_\Omega=4608$, and $K_\Omega=2345$ and set $\ttau = 10^{-3}$. The numerical results are reported in Figs.~\ref{fig:mu1} and~\ref{fig:mu10}. In the case $\mu_-=1$, we observe that the interface rotates into a steady state in which the interfacial fluid continues to circulate along the interface. This motion is commonly referred to as tank-treading. In the case $\mu_-=10$, however, we observe a continuous rotation of the entire vesicle, known as tumbling. These two distinct vesicle dynamics can also be confirmed from the time histories of the discrete quantities. For example, in the first case, the surface fluid kinetic energy tends to level off, whereas in the latter case, it exhibits a periodic behaviour.

\begin{figure}[!htp]
\centering
\includegraphics[width=0.9\textwidth]{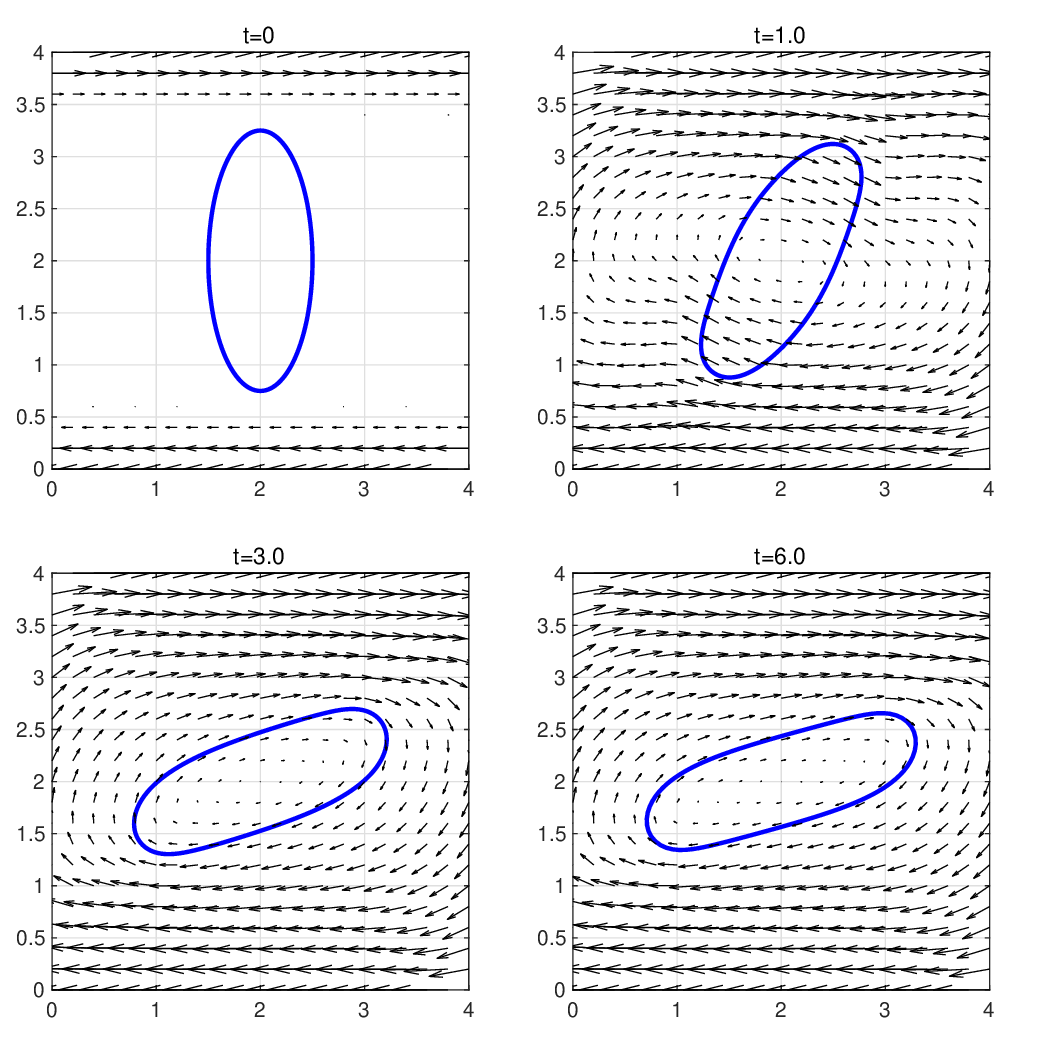}
\includegraphics[width=0.9\textwidth]{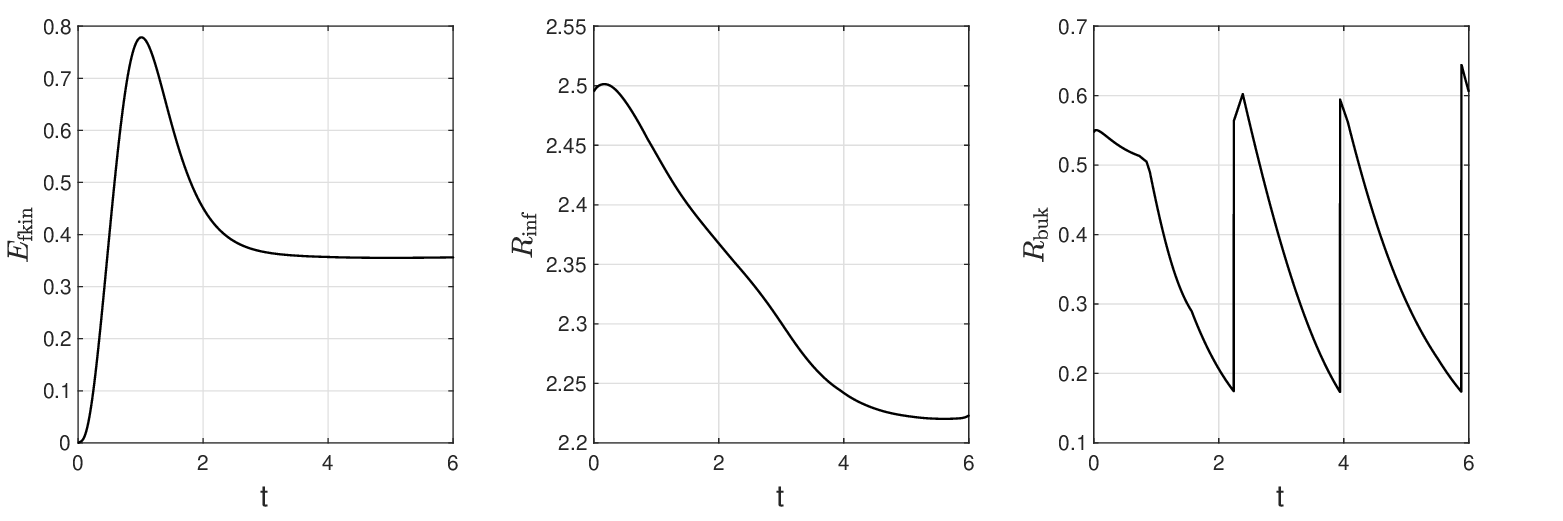}
\caption{The tank treading example in shearing flow for $\mu_-=1$, where we plot $\Gamma^m$ together with velocity fields at $t=0,1,3,6$. Below are the plots of the discrete surface fluid kinetic energy and the mesh quality indicators $R_{\rm inf}$ and $R_{\rm buk}$.}
\label{fig:mu1}
\end{figure}

\begin{figure}[!htp]
\centering
\includegraphics[width=0.9\textwidth]{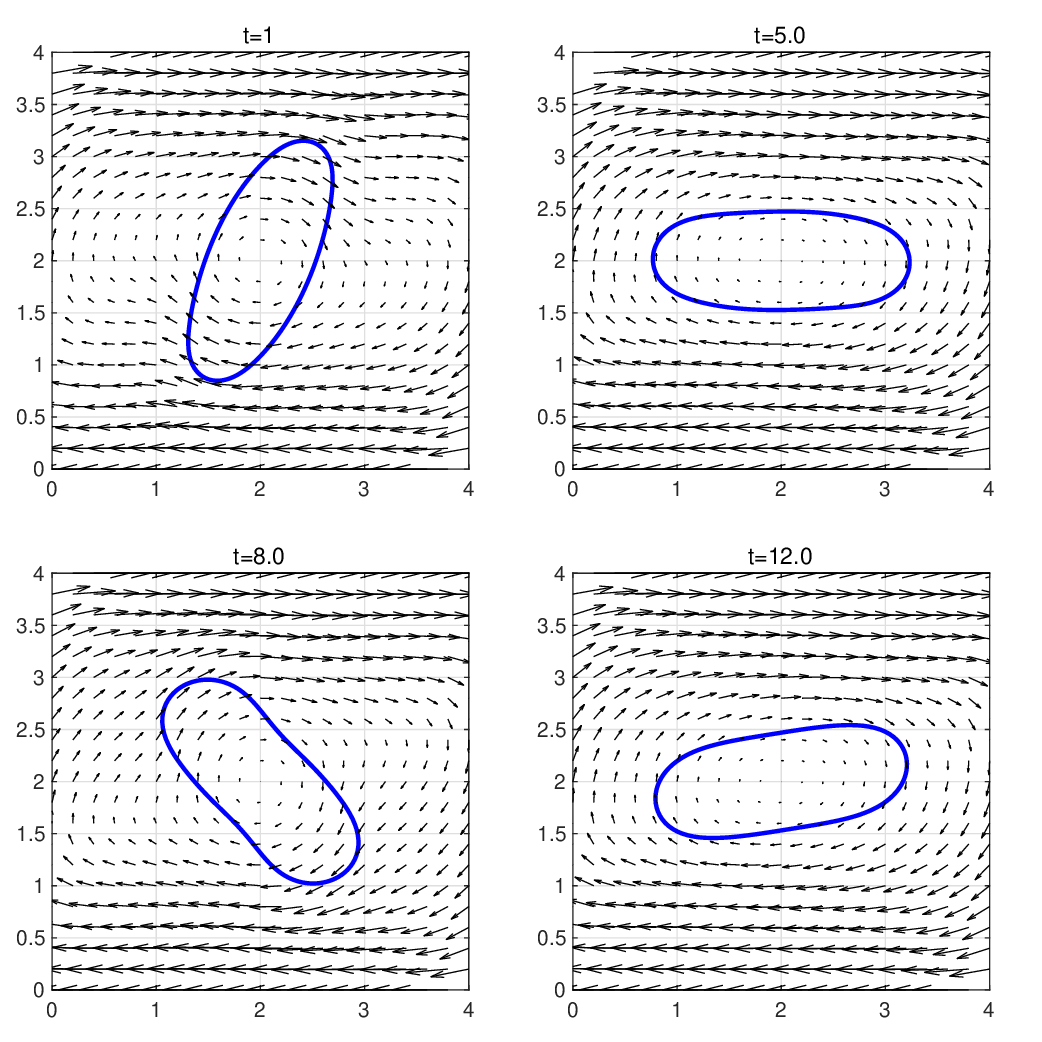}
\includegraphics[width=0.9\textwidth]{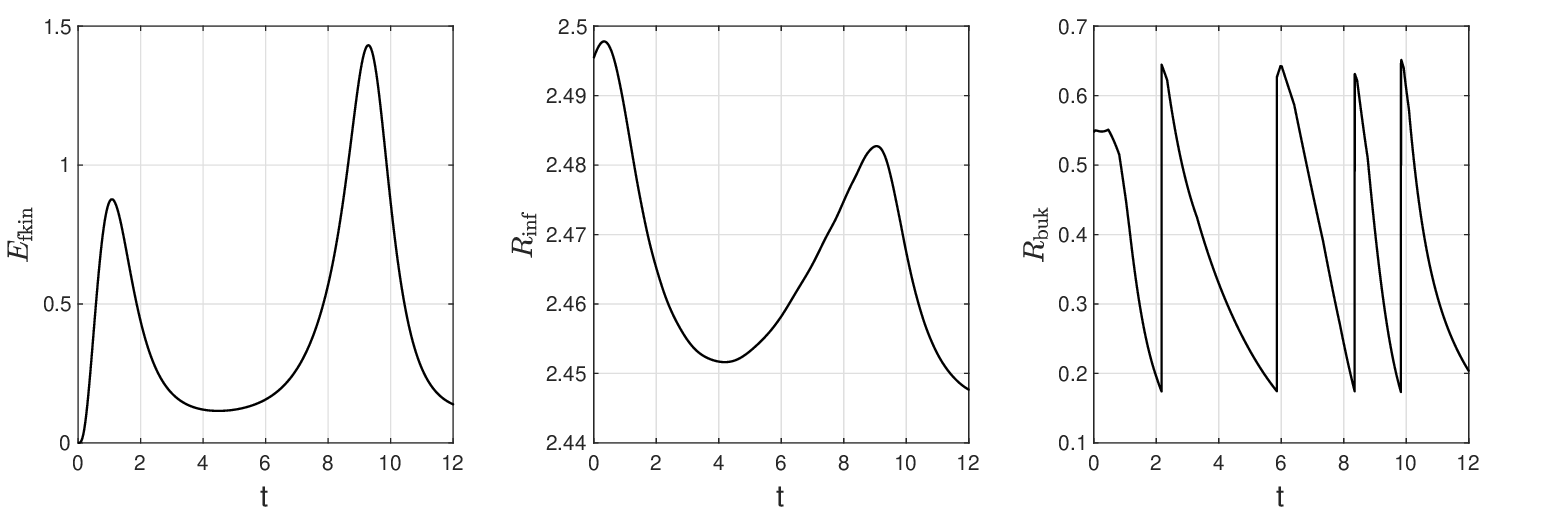}
\caption{The tumbling example in shearing flow for $\mu_-=10$, where we plot $\Gamma^m$ together with velocity fields at $t=0,1,3,6$. Below are the plots of the discrete surface fluid kinetic energy and the mesh quality indicators $R_{\rm inf}$ and $R_{\rm buk}$.}
\label{fig:mu10}
\end{figure}

\section{Unfitted mesh approach}
\label{sec:unf}

\subsection{The unfitted approximations}
It is also possible to consider the approximations in the framework of the unfitted mesh approach. Instead of using the moving mesh strategy described in Section~\ref{sec:dALEmap}, we can alternatively simply fix the bulk mesh velocity $\vec W^m = \vec 0$; hence, $\vec\Phi^m = \vec\id|_{\Omega}$ defined in \eqref{eq:ALEBulk} is the identity map. This allows the surface mesh $\Gamma^m$ to cut the bulk mesh $\mathscr{T}^m$, and so at $t=t_m$ we partition the elements of $\mathscr{T}^m$ into interior, exterior and interface elements as
\begin{equation*}
 \mathscr{T}_{\pm}^m:=\left\{o\in\mathscr{T}^m:\; o\subset\Omega_\pm^m\right\}, \qquad \mathscr{T}_\Gamma^m:=\left\{o\in\mathscr{T}^m:\; o\cap\Gamma^m\neq\emptyset\right\},
\end{equation*}
where $\Omega^m_\pm$ denote the interior and exterior of $\Gm$, as before. Then, we approximate the density $\rho(\cdot, t)$ and viscosity $\mu(\cdot, t)$ via $\rho^m\in S_0^m, \mu^m\in S_0^m$ such that
\begin{equation}
\label{eqn:discreterhovis}
\rho^m|_{o}:=\left\{
\begin{array}{ll}
\rho_-, & \text{if}\ o\in \mathscr{T}_-^m, \vspace{0.15cm}\\
\rho_+, & \text{if}\ o\in\mathscr{T}_+^m,\vspace{0.15cm}\\
\frac{1}{2}(\rho_-+\rho_+), &\text{if}\ o\in \mathscr{T}_{\Gamma}^m,
\end{array}\right.\quad\mbox{and}\quad
\mu^m|_{o}:=\left\{\begin{array}{ll}
\mu_-, &\text{if}\ o\in \mathscr{T}_-^m,\vspace{0.15cm}\\
\mu_+,&\text{if}\ o\in\mathscr{T}_+^m, \vspace{0.15cm}\\
\frac{1}{2}(\mu_-+\mu_+), & \text{if}\ o\in \mathscr{T}_{\Gamma}^m.
\end{array}\right.
\end{equation}

For the unfitted computations, we employ the $\mbox{P2-P1}$ element pair for the discrete velocity and pressure spaces, as defined in \eqref{eq:P2P1}.
For the communication between the bulk and the interface, we use the approaches described in \cite{BGN15stable}, and a bulk mesh adaptation strategy is employed locally in the neighbourhood of the interface (see \cite{BGN15stable}). In addition, for $\mathscr{T}^m$ and the associated finite element spaces $S_k^m(\Omega)$, we introduce the standard interpolation operators $\vec I_k^m: [C(\overline{\Omega})]^d\to [S_k^m(\Omega)]^d$ and the standard projection operator $I_0^m: L^1(\Omega)\to S_0^m(\Omega)$. 

Now, the unfitted finite element method can also be split into two stages. With the same initial data, they are given as follows for each $m\geq 0$.

\noindent
{\bf  At the first stage}, we find $(\vec U^{m+1}, P^{m+1}, P_\Gamma^{m+1}, F_\Gamma^{m+1}, \varkappa^{m+1}, \mathscr{V}^{m+1})\in \bU^m(\vec g)\times\bP^m\times [V^h(\Gm)]^4$ and $P_{\rm sing}^{m+1}\in\bR$ by solving
\begin{subequations}\label{eqn:ufdnv}
\begin{align}
\label{eq:ufdnv1}
&\frac{1}{2\,\ttau}\left[\ipD{\rho^m\,\vec U^{m+1}, ~\vec\xi^h} + \ipD{I_0^m\rho^{m-1}\vec U^{m+1}, ~\vec\xi^h}\right]-\ipD{P^{m+1},~\nabla\cdot\vec\xi^h}\nn\\
&\qquad\quad+ 2\ipD{\mu^m\,\mat{D}(\vec U^{m+1}),~\mat{D}(\vec\xi^h)}-P_{\rm sing}^{m+1}\ipd{\vec\nu_p^m,~\vec\xi^h}_{\Gm}^h +\mathscr{A}\left(\rho^m\,\vec I_2^m\vec U^m;\vec U^{m+1},\vec\xi^h\right) \nn\\
&\qquad\quad+\frac{\rho_\Gamma}{2\,\ttau}\left[\ipd{\vec U^{m+1},~\vec\xi^h}^h_{\Gm}+ \ipd{\Pi_m^{m-1}\vec U^{m+1},~\Pi_m^{m-1}\vec\xi^h}^h_{\Gamma^{m-1}}\right]\nn\\
&\qquad\quad- \ipd{P^{m+1}_\Gamma,~\nabs\cdot[\vec\pi_1^m\vec\xi^h]}_{\Gm}    +2\mu_\Gamma\ipd{\mat{D}_s^m(\vec\pi_1^m\vec U^{m+1}),~\mat{D}_s^m(\vec\pi_1^m\vec\xi^h)}_{\Gm}\nn\\
&\qquad\quad = \frac{1}{\ttau}\ipD{I_0^m\rho^{m-1}\,\vec I_2^m\vec U^m,~\vec\xi^h} + \frac{\rho_\Gamma}{\ttau}\ipd{\vec I_2^m\vec U^m,~\Pi_m^{m-1}\vec\xi^h}_{\Gamma^{m-1}}^h  + \alpha\ipd{F_\Gamma^{m+1}\vec\nu_p^m,~\vec\xi^h}_{\Gm}^h\nn\\
&\qquad\qquad -\frac{1}{2}\ipd{\rho_+\vec U^m\cdot\vec n_{_{\partial\Omega}},~\vec U^m\cdot\vec\xi^h\,}_{\partial_2\Omega}\qquad\forall\vec\xi^h\in\bU^m(\vec 0); \\[0.5em]
\label{eq:ufdnv2}
&\ipD{\nabla\cdot\vec U^{m+1},~q^h} = 0\qquad\forall q^h\in \bP^m,\qquad\mbox{and}\quad \ipd{\vec U^{m+1}, \vec\nu_p^m}_{\Gm}^h=0;\\[0.5em]
&\ipd{\nabs\cdot[\vec\pi_1^m\vec U^{m+1}],~q^h_{\Gamma}}_{\Gamma^m}=0\qquad\forall q^h_\Gamma\in V^h(\Gm);
\label{eq:ufdnv3}
\end{align}
\end{subequations}
together with \eqref{eq:fdsf1}, \eqref{eq:fdsf2} and
\begin{equation}
\ipd{\mathscr{V}^{m+1},~\zeta^h}_{\Gm}^h-\ipd{\vec U^{m+1}\cdot\vec\nu_p^m,~\zeta^h}_{\Gm}^h=0\qquad\forall\zeta^h\in V^h(\Gm).
\label{eq:ufdsf3}
\end{equation}

\noindent
{\bf   At the second stage}, given $\mathscr{V}^{m+1}\in V^h(\Gm)$, we seek $\vec X^{m+1}\in [V^h(\Gm)]^d$, $\lambda^{m+1}_\Gamma\in V^h(\Gm)$ and $\vec U_\Gamma^{m+1}\in[V^h(\Gm)]^d$ such that
\begin{subequations}\label{eqn:uXfull}
 \begin{align}
 \label{eq:uXfull1}
&\ipd{\frac{\vec X^{m+1}-\vec\id}{\ttau}\cdot\vec\nu_p^m,~\eta^h}_{\Gm}^h = \ipd{\mathscr{V}^{m+1},~\eta^h}_{\Gm}^h\qquad\forall\eta^h\in V^h(\Gm),\\
\label{eq:uXfull2}
&\ipd{\vec U_\Gamma^{m+1}, ~\vec\eta^h}_{\Gm}^h = \ipd{\vec U^{m+1}, \vec\eta^h}_{\Gm}\qquad\forall\vec\eta^h\in[V^h(\Gm)]^d,\\
\label{eq:uXfull3}
&\ipd{\mat{P}_{\Gm}\frac{\vec X^{m+1}-\vec\id}{\ttau},~\vec\eta^h}_{\Gm}^h + \ipd{\lambda^{m+1}_\Gamma\,\vec\nu_p^m,~\vec\eta^h}_{\Gm}^h
= \ipd{\mat{P}_{\Gm}\vec U_\Gamma^{m+1},~\vec\eta^h}_{\Gm}^h\quad\forall\vec\eta^h\in[V^h(\Gm)]^d.
 \end{align}
\end{subequations}

While the unfitted finite element discretization of the bulk-surface Navier--Stokes system in \eqref{eqn:ufdnv} is aligned with \cite{BGN16sfluidic,BGN17fluidic}, the approximation of the curvature force $F_\Gamma^{m+1}$ differs and constitutes a key distinction. In contrast to \cite{BGN16sfluidic,BGN17fluidic}, the above fully discrete approximation can be shown to be unconditionally energy-stable. In fact, similarly to the presented fitted ALE approximations in Section~\ref{sec:pro}, we have the following theorems for the properties of the unfitted finite element approximations, and the proofs are omitted for brevity.
\begin{thm}[existence and uniqueness]
Assuming that the $LBB_\Gamma$ inf-sup stability condition \eqref{eq:infsup} holds, then the linear system \eqref{eqn:ufdnv}, \eqref{eq:fdsf1}, \eqref{eq:fdsf2} and \eqref{eq:ufdsf3} admits a unique solution $(\vec U^{m+1}, P^{m+1}, P_\Gamma^{m+1}, P_{\rm sing}^{m+1}, F_\Gamma^{m+1}, \varkappa^{m+1}, \mathscr{V}^{m+1})$.
Moreover, assuming that the vertex normals in \eqref{eq:nuhomegah} satisfy
\begin{equation}
\vec\nu_p^m(\vec q)\neq 0\quad\forall \vec q\in Q_\Gamma^m.\label{eq:vnormalassump1}
\end{equation}
then the linear system \eqref{eqn:uXfull} admits a unique solution $(\vec X^{m+1}, \vec U_\Gamma^{m+1}, \lambda^{m+1}_\Gamma)$ with $\lambda^{m+1}_\Gamma=0$.
\end{thm}

\begin{thm}[stability estimate] Let $(\vec U^{m+1}, P^{m+1}, P_\Gamma^{m+1}, P_{\rm sing}^{m+1}, F_\Gamma^{m+1}, \varkappa^{m+1}, \mathscr{V}^{m+1})$ be a solution to the linear system \eqref{eqn:ufdnv}, \eqref{eq:fdsf1}, \eqref{eq:fdsf2} and \eqref{eq:ufdsf3} for each $m\geq 0$. In the case $\vec g = \vec 0$, the discrete solution satisfies the following energy stability estimate:
\begin{align}\label{eq:unfdES}
&\bar{\mathcal{E}}(\rho^m, \vec U^{m+1}, \Gamma^m, \varkappa^{m+1})+ 2\ttau\norm{\sqrt{\mu^m}\mat{D}(\vec U^{m+1})}_0^2+ 2\ttau\,\mu_\Gamma\norm{\mat{D}_{s}^m(\vec\pi_1^m\vec U^{m+1})}_{\Gm,0}^2\nn\\
&\quad+ \frac{\rho_+\ttau}{2}\ipd{\vec U^m\cdot\vec n,~\vec U^m\cdot\vec U^{m+1}}_{\partial_2\Omega}\leq\bar{\mathcal{E}}(I_0^m\rho^{m-1}, \vec I_2^m\vec U^{m}, \Gamma^{m-1}, \varkappa^{m}),
\end{align}
where we define
\begin{equation*}
\bar{\mathcal{E}}(\rho, \vec u, \Gamma, \varkappa):= \frac{1}{2}(\rho\vec u,~\vec u) + \frac{1}{2}\ipd{\rho_\Gamma\vec u,~\vec u}_{\Gamma}^h + \frac{\alpha}{2}\ipd{\varkappa-\bkap,~\varkappa-\bkap}_{\Gamma}.
\end{equation*}
\end{thm}

\subsection{Numerical results}

For the bulk mesh strategy, we refer the reader to \cite{BGN15stable} with the notation ``$n{\rm adapt}_{k,l}$'' to denote $\ttau = 10^{-3}/n,  N_f = 2^k$ and $N_c= 2^l$. Unless otherwise stated, we always use the discretization parameters $5{\rm adapt}_{9,4}$.

\noindent
{\bf Example 4}: In this example, we repeat the simulations from Example~2 in the fitted approach and consider the evolution of a smooth letter ``C'' in $\Omega=(0,2)^2$. The numerical results are visualized in Figs.~\ref{fig:uc0} and~\ref{fig:uc1} for $\rho_\Gamma=0$ and $\rho_\Gamma=1$, respectively. Here, we observe vesicle behaviours similar to those in the ALE results. We also observe good preservation of volume and area, as well as the decay of the discrete energy. 

\begin{figure}[!htp]
\centering
\includegraphics[width=0.8\textwidth]{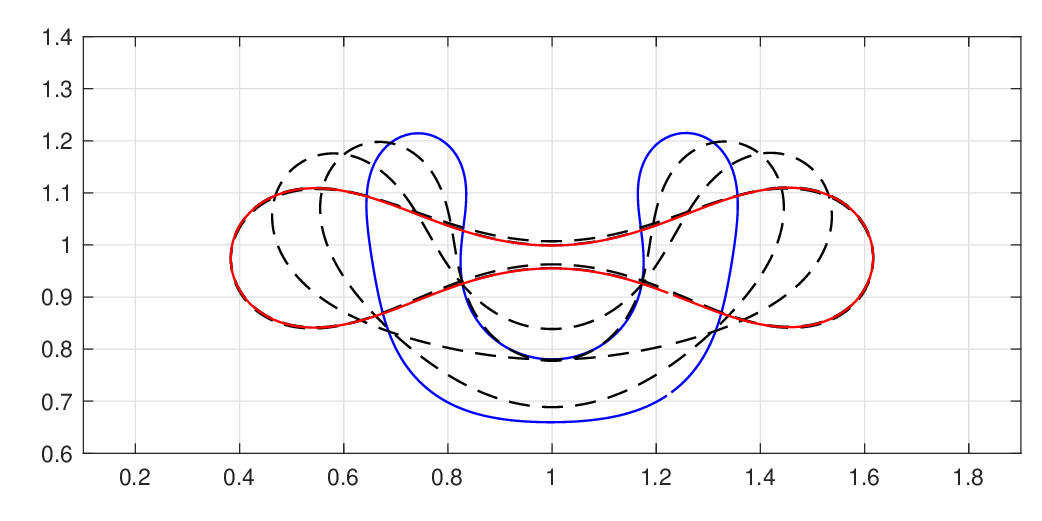}
\includegraphics[width=0.95\textwidth]{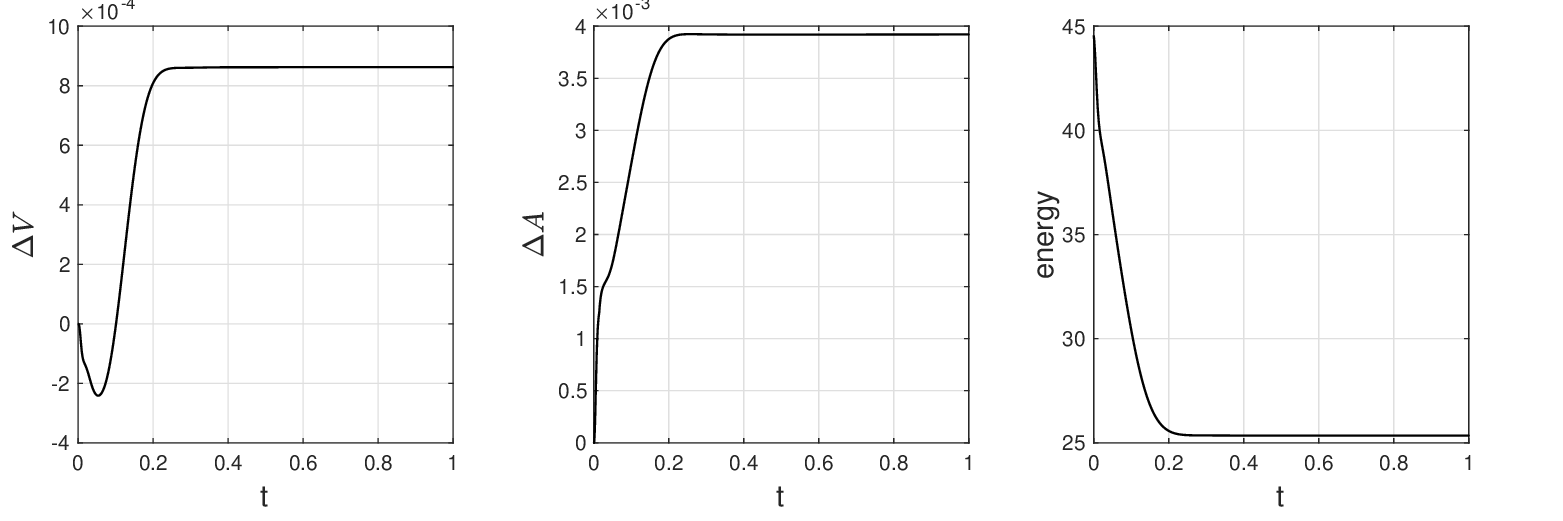}
\caption{Evolution for a smooth letter ``C'' in the case of $\rho_\Gamma=0$ towards the steady state (red solid line), where the plots of $\Gamma^m$ are shown at $t=0, 0.05, 0.3, 0.6, 1.0$. Below are plots of the discrete quantities.}
\label{fig:uc0}
\end{figure}

\begin{figure}[!htp]
\centering
\includegraphics[width=0.8\textwidth]{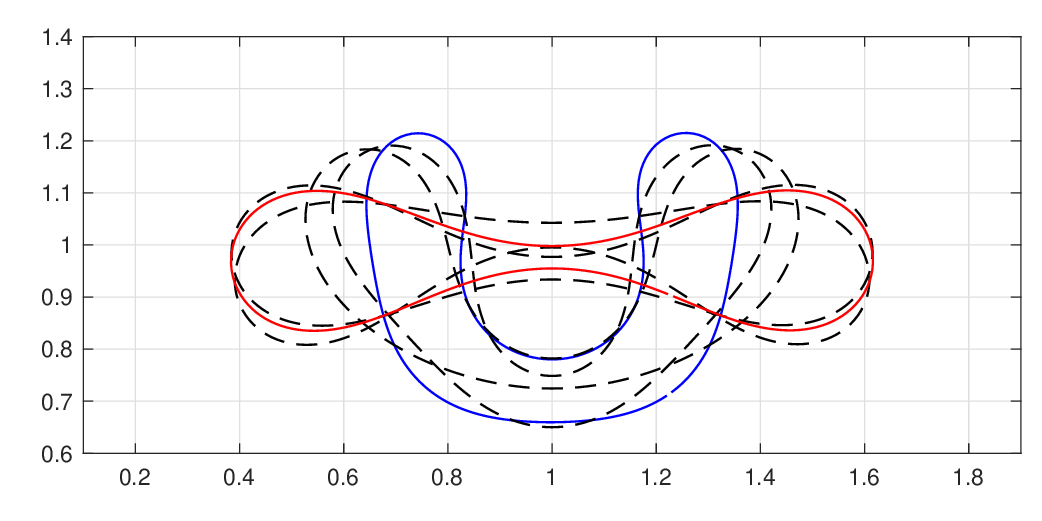}
\includegraphics[width=0.95\textwidth]{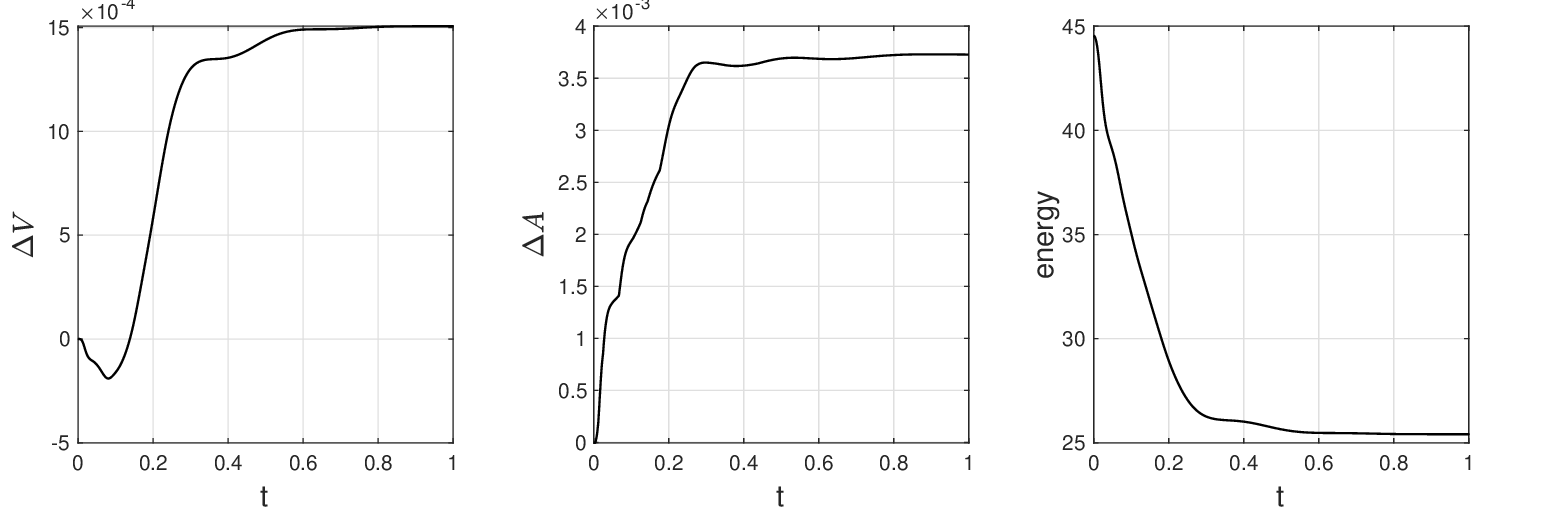}
\caption{Evolution for a smooth letter ``C'' in the case of $\rho_\Gamma=1$ towards the steady state (red solid line), where the plots of $\Gamma^m$ are shown at $t=0, 0.05, 0.3, 0.6, 1.0$. Below are plots of the discrete quantities.}
\label{fig:uc1}
\end{figure}

\noindent
{\bf Example 5}: In our final example, we conduct an experiment for the evolution of a vesicle flowing through a constriction; see \cite[Fig.~4]{BGN16sfluidic}. We consider
$\Omega=(-2,2)\times(-1,1)\setminus ([-1,1]\times[0.5,1])$
with $\partial_2\Omega=\{2\}\times(-1,1)$. On the left boundary $\{-2\}\times(-1,1)$, we prescribe the inhomogeneous boundary condition $\vec g(\vec z)=(1-z_2^2, 0)^T$ to model Poiseuille flow. Initially, the vesicle is given as an elongated tube of total dimension $0.2\times 1.5$. We choose the parameters
\begin{equation*}
\rho_\pm = 0,\quad\mu_\pm=1,\quad \rho_\Gamma=15,\quad \mu_\Gamma=1, \quad \alpha=0.1.
\end{equation*}
Owing to the large-scale displacement of the vesicle, our fitted mesh method becomes less effective for this particular example, whereas the unfitted mesh approach possesses clear advantages. The results are reported in Fig.~\ref{fig:ucon}. In this example, we reset the evolving curvature $\varkappa^m$ to the ``BGN'' curvature $\kappa^m$ defined on $\Gamma^m$ analogously to \eqref{eqn:ic} every $N=1000$ time steps, so for a total of twelve times, to reduce the accumulation of errors in the discrete curvature approximations.

\begin{figure}
\centering
\includegraphics[width=0.9\textwidth]{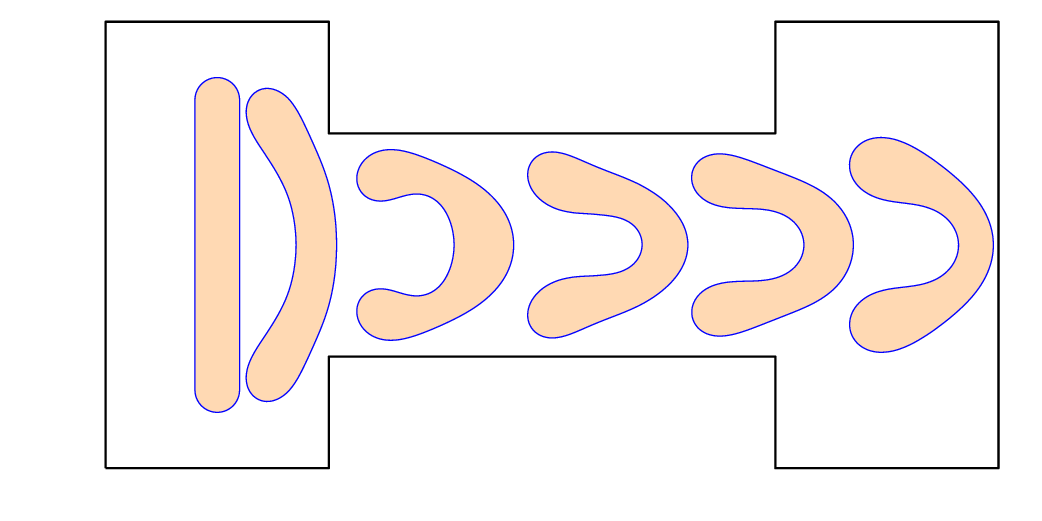}
\includegraphics[width=0.9\textwidth]{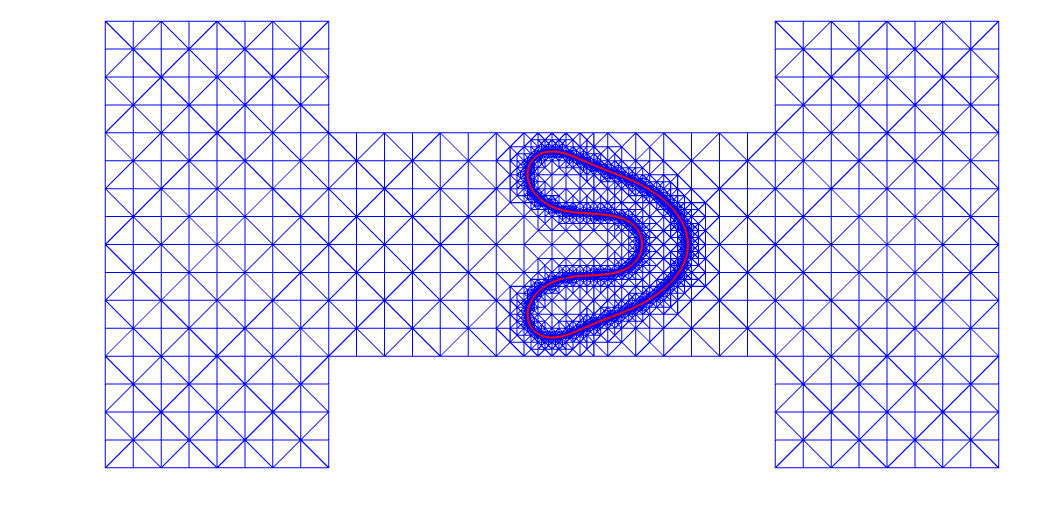}
\includegraphics[width=0.9\textwidth]{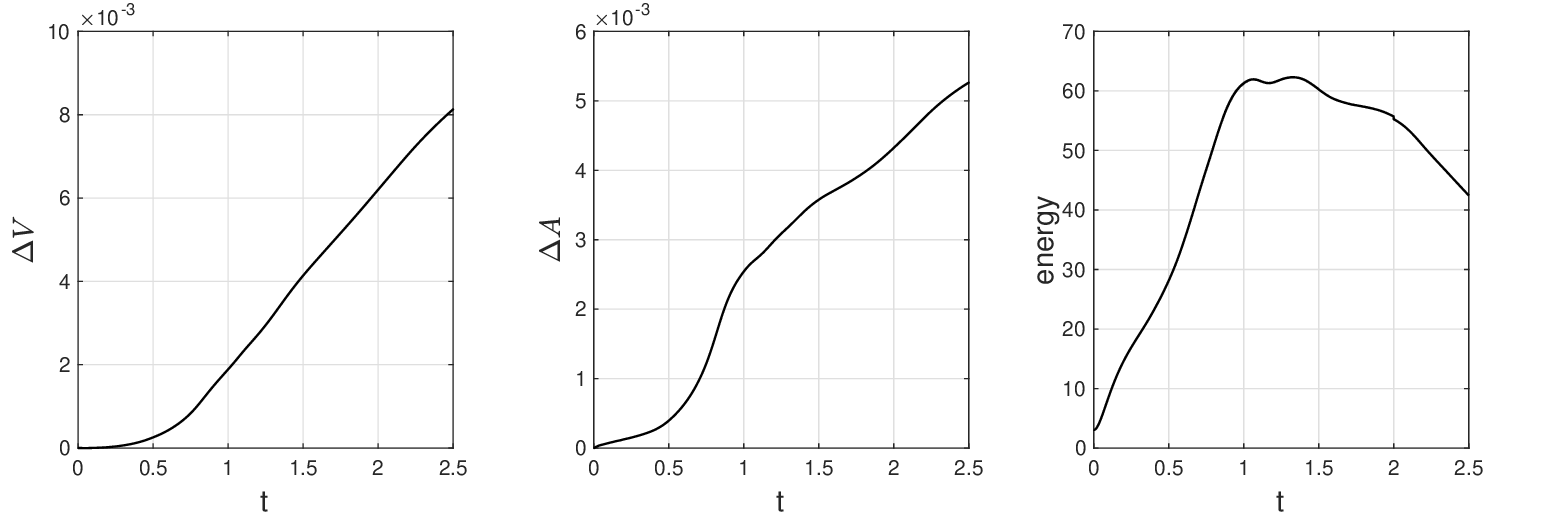}
\caption{Simulation of a vesicle flowing through a constriction. Upper panel: plots of the interface $\Gamma^m$ at $t=0, 0.5, 1.0, 1.5, 2.0, 2.5$. Middle panel: plots of the computational meshes at $t=1.5$. Lower panel: the time history of the discrete quantities.}
\label{fig:ucon}
\end{figure}

\section{Conclusions}\label{sec:con}

The present work introduced a unified, energy-stable finite element framework for the dynamics of fluidic biomembranes that consistently couples membrane bending, interfacial surface fluidity, and bulk fluid interactions while preserving a discrete energy dissipation law. Central to the approach is a unified weak formulation of the governing dynamical system that permits independent choices of the bulk mesh velocity and the tangential velocity of the interface mesh. Based on this formulation, we proposed fully stable finite element discretizations in both fitted and unfitted mesh settings. Numerical experiments are presented to demonstrate the robustness, favourable properties, and effectiveness of the proposed methods in simulating complex biomembrane phenomena.

\section*{Acknowledgments}
This work was partially supported by the National Natural
Science Foundation of China No. 12401572 (Q.Z) and the Key Project of the National Natural Science Foundation of China No.~12494555 (Q.Z).

\begin{appendices}
\section{Differential calculus}
For simplicity, we use $\Omega_\pm$ to denote $\Omega_\pm(t)$. It follows from the Reynolds transport theorem, the divergence free condition \eqref{eq:div} and \eqref{eq:dTd} that
\begin{align}
\ddt\bigl(\vec u,~\vec\chi\bigr)_{\Omega_-}&=\bigl(\partial_t^\bullet\vec u,~\vec\chi\bigr)_{\Omega_-}+ \bigl(\partial_t^\bullet\vec\chi,~\vec u\bigr)_{\Omega_-} + \bigl(\vec u\cdot\vec\chi,~\nabla\cdot\vec u\bigr)_{\Omega_-}\nn\\
&= \bigl(\Dtw\vec u,~\vec\chi\bigr)_{\Omega_-}+ \bigl(\Dtw\vec\chi,~\vec u\bigr)_{\Omega_-}+ \bigl([\vec u-\vec w],~\nabla[\vec u\cdot\vec\chi]\bigr)_{\Omega_-}.\label{eq:aleneg}
\end{align}
Using again the Reynolds transport theorem and recalling \eqref{eq:ALET}, we have
\begin{align}
&\ddt\bigl(\vec u,~\vec\chi\bigr)_{\Omega_+} = \bigl(\partial_t\vec u,~\vec\chi\bigr)_{\Omega_+}+\bigl(\vec  u,~\partial_t\vec\chi\bigr)_{\Omega_+}- \ipd{\vec u\cdot\vec\chi,~\vec u\cdot\vec\nu}_{\Gt}\nn\\
&=\bigl(\Dtw\vec u,~\vec\chi\bigr)_{\Omega_+} + \bigl(\vec u,~\Dtw\vec\chi\bigr)_{\Omega_+}-\bigl(\vec w,~\nabla[\vec u\cdot\vec\chi]\bigr)_{\Omega_+}-\ipd{\vec u\cdot\vec\chi,~\vec u\cdot\vec\nu}_{\Gt}\nn\\
&=\bigl(\Dtw\vec u,~\vec\chi\bigr)_{\Omega_+} + \bigl(\vec u,~\Dtw\vec\chi\bigr)_{\Omega_+} + \bigl([\vec u-\vec w],~\nabla[\vec u\cdot\vec\chi]\bigr)_{\Omega_+}\nn\\
&\qquad\qquad\qquad -\bigl(\vec u,~\nabla[\vec u\cdot\vec\chi]\bigr)_{\Omega_+}-\ipd{\vec u\cdot\vec\chi,~\vec u\cdot\vec\nu}_{\Gt}\nn\\
&=\bigl(\Dtw\vec u,~\vec\chi\bigr)_{\Omega_+} + \bigl(\vec u,~\Dtw\vec\chi\bigr)_{\Omega_+} + \bigl([\vec u-\vec w],~\nabla[\vec u\cdot\vec\chi]\bigr)_{\Omega_+}-\ipd{\vec u\cdot\vec n_{_{\partial\Omega}},~\vec u\cdot\vec\chi}_{\partial_2\Omega},\label{eq:alepos}
\end{align}
where the last equality results from integration by parts. Now multiplying \eqref{eq:aleneg}, \eqref{eq:alepos} with $\rho_-$ and $\rho_+$, respectively, and then combining the two equations leads to
\begin{align}
&\ddt\bigl(\rho\,\vec u,~\vec\chi\bigr)=\sum_{\pm}\rho_\pm \ddt\bigl(\vec u,~\vec\chi\bigr)_{\Omega_\pm}\nn\\
&=\bigl(\rho\,\Dtw\vec u,~\vec\chi\bigr) +\bigl(\rho\,\vec u,~\Dtw\vec\chi\bigr)+\bigl(\rho\,[\vec u-\vec w],~\nabla[\vec u\cdot\vec\chi]\bigr) - \rho_+\ipd{\vec u\cdot\vec n_{_{\partial\Omega}},~\vec u\cdot\vec\chi}_{\partial_2\Omega},\nn
\end{align}
which gives
\begin{align}
\bigl(\rho\,\Dtw\vec u,~\vec\chi\bigr) =& \ddt\bigl(\rho\,\vec u,~\vec\chi\bigr) - \bigl(\rho\,\vec u,~\Dtw\vec\chi\bigr)+ \rho_+\ipd{\vec u\cdot\vec n_{_{\partial\Omega}},~\vec u\cdot\vec\chi}_{\partial_2\Omega}\nn\\
&-\bigl(\rho\,[\vec u-\vec w]\cdot\nabla\vec u,~\vec\chi\bigr)-\bigl(\rho\,[\vec u-\vec w]\cdot\nabla\vec\chi,~\vec u\bigr).
\label{eq:ale1}
\end{align}

It follows from the Reynolds transport theorem, the surface divergence free condition \eqref{eq:if3}, and \eqref{eq:fdTd} that
\begin{align}
\ddt\ipd{\vec u,~\vec\chi}_{\Gt}&=\ipd{\partial_t^\bullet\vec u,~\vec\chi}_{\Gt} + \ipd{\vec u,~\partial_t^\bullet\vec\chi}_{\Gt}\nn\\
&=\ipd{\Dtv\vec u,~\vec\chi}_{\Gt} + \ipd{\vec u,~\Dtv\vec\chi}_{\Gt} + \ipd{\vec u-\vec v,~\nabs[\vec u\cdot\vec\chi]}_{\Gt},\nn
\end{align}
which leads to
\begin{align}
\ipd{\Dtv\vec u,~\vec\chi}_{\Gt} =& \ddt\ipd{\vec u,~\vec\chi}_{\Gt}-\ipd{\vec u,~\Dtv\vec\chi}_{\Gt}\nn\\
&\quad - \ipd{[\vec u-\vec v]\cdot\nabs\vec u,~\vec\chi}_{\Gt}- \ipd{[\vec u-\vec v]\cdot\nabs\vec\chi,~\vec u}_{\Gt}.\label{eq:fale1}
\end{align}
\end{appendices}

\footnotesize
\bibliographystyle{model1b-num-names}
\bibliography{bib}


\end{document}